\documentclass[a4paper, 11pt, twoside, reqno]{amsart}
  \usepackage[a4paper,inner=2cm,outer=2cm,top=4cm,bottom=3.5cm]{geometry}
\usepackage{amsmath,amscd}
\usepackage{amssymb}
 \usepackage{amsthm}
\usepackage{comment}
\usepackage{mathrsfs}
\usepackage{graphicx, xcolor}
\usepackage{xcolor}
\usepackage{mathtools}
\usepackage{mathrsfs}
\usepackage[normalem]{ulem}
\usepackage[ocgcolorlinks, linkcolor=blue,citecolor=red]{hyperref}
\usepackage{orcidlink}

\usepackage{bm}
\usepackage{bbm}
\usepackage{url}

\usepackage[utf8]{inputenc}
\usepackage{mathtools,amssymb}
\usepackage{esint}

    \usepackage{amssymb}
\usepackage{tikz}
\usepackage{dsfont}
\usepackage{relsize}
\usepackage{url}
\usepackage{xcolor}
\usepackage{graphicx}
\usepackage{mathrsfs}
\usepackage[shortlabels]{enumitem}
\usepackage{lineno}
\usepackage{amsmath}
\usepackage{enumitem}
\usepackage{amsthm} 
\usepackage{verbatim}
\usepackage{dsfont}
\usepackage{tikz}
\numberwithin{equation}{section}

\allowdisplaybreaks

\mathtoolsset{showonlyrefs}

\graphicspath{{images/}}
\theoremstyle{plain}

\newtheorem{theorem}{Theorem}[section]
\newtheorem{lemma}[theorem]{Lemma}

\newtheorem{proposition}[theorem]{Proposition}

\newtheorem{remark}[theorem]{Remark}
\allowdisplaybreaks[0]

\title[ Semilinear wave equation]{Stable determination of damping and potential coefficients in a semilinear wave equation}

\author[R. Bhardwaj]{Rahul Bhardwaj\,\orcidlink{0009-0007-5122-7781}}
\author[M. Kumar]{Mandeep Kumar\,\orcidlink{0009-0004-9577-7759}}
\author[M. Vashisth]{Manmohan Vashisth\,\orcidlink{0000-0002-3417-4055}}
\address{{Department of Mathematics, Indian Institute of Technology Ropar, Rupnagar, Punjab-140001, INDIA.}}

\email{bhardwaj161067@gmail.com}
\email{mandeep.sansanwal@gmail.com}
\email{manmohanvashisth@iitrpr.ac.in}

\newcommand{\R}{{\mathbb R}}

\newcommand {\p} {\partial}

\newcommand{\norm}[1]{\lVert #1 \rVert}
\DeclareMathOperator{\supp}{supp} 

\begin{document}
	\begin{abstract}
We consider an inverse problem for a semilinear wave equation with time-independent damping, linear potential, and nonlinear potential coefficients in a bounded domain of $\mathbb{R}^n$ for $n\geq 2$. The main objective is to establish stability estimates for the simultaneous recovery of these coefficients from the associated Dirichlet-to-Neumann map. Our approach combines second-order linearization with suitably constructed geometric optics and asymptotic solutions. We establish H\"older-type stability estimates for the recovery of 
each of the three coefficients appearing in the semilinear wave equation under suitable a priori bounds on these coefficients. To the best of our knowledge, this is the first stability  result for simultaneous determination of time-independent damping, linear and nonlinear potentials in a semiliner  wave equation.

		

	\end{abstract}
	\maketitle
  \section{Introduction}
  
\subsection{Motivation and formulation of the inverse problem}

Inverse boundary value problems for hyperbolic equations seek to determine
unknown properties of a medium from measurements performed at its boundary.
These kind of  problems arise naturally in wave propagation, where spatially varying
coefficients may describe various physical quantities such as   dissipation, restoring forces, material
inhomogeneities, or nonlinear responses of the underlying medium.

In this paper, we study the stable simultaneous determination of the
coefficients appearing in a quadratic semilinear damped wave equation. More precisely,
let $\Omega\subset\mathbb{R}^n$, $n\geq2$, be a bounded connected domain
with smooth boundary $\partial\Omega$, and   for a fixed $T>0$,  we denote by $\Omega_T := (0,T)\times\Omega$, a space-time domain and $\Sigma := (0,T)\times\partial\Omega$, a lateral boundary of $\Omega_T$.
An inverse problem related to the  following  initial-boundary value problem (IBVP) for a semilinear wave equation with lower-order terms is  considered in the present paper.
\begin{equation}\label{equation; IBVP}
\begin{cases}
\Box u(t,x)+a(x)\partial_tu(t,x)+b(x)u(t,x)+q(x)u^2(t,x)=0,
&(t,x)\in\Omega_T,\\[1mm]
u(t,x)=f(t,x),
&(t,x)\in\Sigma,\\[1mm]
u(0,\cdot)=0,\qquad \partial_tu(0,\cdot)=0,
&~~~~~~~~~\quad x\in\Omega,
\end{cases}
\end{equation}

where
    $\Box:=\partial_t^2-\Delta_x$
denotes the standard wave operator. The
coefficient $a=a(x)$ represents the damping, while $b=b(x)$ and $q=q(x)$
denote  the linear and nonlinear potentials respectively. 

For sufficiently small Dirichlet data $f$ (see Theorem \ref{main_thm:welposedness}), let $u_f$ denote the unique
solution of \eqref{equation; IBVP}. Then the associated nonlinear
Dirichlet-to-Neumann (DN) map is given by
\begin{equation}\label{eq:DN-map}
    \Lambda_{a,b,q}(f)
    :=
    \left.\partial_\nu u_f\right|_{\Sigma},
\end{equation}
where
    $\partial_\nu u_f
    :=
    \nu(x)\cdot\nabla_xu_f$
and $\nu$ denotes the outward unit normal vector to $\partial\Omega$.
The precise class of sufficiently small admissible boundary values is
introduced in Section~\ref{sec:forward-problem}.

The inverse problem considered here is to determine the three coefficients
$a$, $b$, and $q$ from boundary measurements given by
\eqref{eq:DN-map}. Our main interest is the stability of the coefficients that is, we want to bound the difference between two sets of coefficients in terms of the difference between the corresponding boundary measurements.
\subsection{Physical background}

The hyperbolic nature of the principal part in the IBVP \eqref{equation; IBVP} reflects
fundamental properties of wave motion, in particular causality and
finite propagation speed, whereas the lower-order terms encode different
properties of the underlying medium. The damping
coefficient $a(x)$ accounts for spatially dependent attenuation or energy
dissipation, while the linear potential $b(x)$ represents a
position-dependent linear interaction. The term
   $ q(x)u^2$
describes a spatially varying nonlinear response of the medium. In particular,
the coefficient $q(x)$ determines the spatial strength of the nonlinear
interaction, while the power square describes its dependence on the wave
amplitude.

Equations of this form may be regarded as damped semilinear variants of the
Klein--Gordon equation and arise in mathematical models of wave propagation
in heterogeneous and dissipative media; see, for example,
\cite{morawetz1968time}. Related mechanisms appear in acoustics, elasticity,
seismology, and other areas in which attenuation and nonlinear material
responses influence the observed wave field.

From the inverse problem perspective, the coefficients $a$, $b$, and $q$
describe properties in the inaccessible interior of $\Omega$, whereas waves
can be generated and observed at the boundary. This naturally leads to the
question of whether these interior coefficients can be determined from the
map $\Lambda_{a,b,q}$. Since measured data are necessarily affected by
perturbations, a stability estimate is of particular interest.

\subsection{Admissible coefficients and the main result}
Fix an integer $m>n+1$, and let $K>\frac{n}{2}+3m+4$. For $M>0$, define
\begin{equation}\label{eq:admissible-coefficients}
\mathcal A(M)
:=
\left\{
(a,b,q)\in
\bigl(C_c^\infty(\Omega)\bigr)^3:
\|a\|_{H^K(\Omega)}
+
\|b\|_{H^K(\Omega)}
+
\|q\|_{H^K(\Omega)}
\leq M
\right\}.
\end{equation}

For sufficiently small $\delta>0$, let
$\mathcal D_{m+1}^{\delta}$ denote the admissible class of small boundary
values defined in \eqref{E-delta}. We denote the supremum of the difference between two DN maps over all admissible boundary data by
\begin{equation}\label{eq:def-eta-main}
    \eta
    :=
    \sup_{f\in\mathcal D_{m+1}^{\delta}}
    \|\left(\Lambda_{a_1,b_1,q_1}-\Lambda_{a_2,b_2,q_2}\right)(f)\|_{L^2(\Sigma)}.
\end{equation}

Note that $\eta$ is finite by Theorem
\ref{main_thm:welposedness}. Throughout this paper, $C>0$ denotes a generic constant that is independent of the parameters involved in the estimates and whose value may change from line to line. We are now ready to state the main result of this paper.

\begin{theorem}\label{thm:main-stability}
Let m be an integer with $m>n+1$, $M>0$, and let
$(a_j,b_j,q_j)\in\mathcal{A}(M)$ for $j=1,2.$
Assume that $T>\operatorname{diam}(\Omega)$. Then there exist constants
$C>0, \alpha\in(0,1),$ depending only on $\Omega$, $T$, and $M$ such that, for all $\eta\geq0$, we have
\begin{align}\label{eq:main-stability-a}
\|a_1-a_2\|_{L^\infty(\Omega)} + \|b_1-b_2\|_{L^\infty(\Omega)} +\|q_1-q_2\|_{L^\infty(\Omega)}
\leq
C\eta^\alpha.
\end{align}
In particular, if $\Lambda_{a_1,b_1,q_1}(f) = \Lambda_{a_2,b_2,q_2}(f)$ for all $f\in D^{\delta}_{m+1}$, then $\left(a_1,b_1,q_1\right) = \left(a_2,b_2,q_2\right)$ in $\Omega.$  
\end{theorem}

The present work builds on the recent reconstruction result of Bhardwaj
et al.~\cite{bhardwaj2026reconstructionpotentialdampingcoefficients}, where
the corresponding inverse problem for the IBVP~\eqref{equation; IBVP} was
studied from the perspective of reconstruction. Here, we
address the quantitative stability problem and derive H\"older-type estimates for the recovery of all three coefficients from the nonlinear DN map. Our result is also related to the work of Lassas et al.
\cite{Lassas2022UniquenessRA}, where a semilinear wave equation with a
time-dependent nonlinear potential was studied in the absence of damping and
a linear potential. For time-dependent coefficients, the finite speed of
propagation implies that the DN map cannot recover the coefficients in the
entire space-time domain see, for instance, \cite[Section~1]{Kian2017}.
Therefore, recovery is restricted to a suitable accessible region.

In the present work, the coefficients are time-independent, and the condition $T>\operatorname{diam}(\Omega)$ allows us to obtain stability estimates throughout $\Omega$. An additional difficulty arises from the presence of the damping term. In contrast to the setting of
\cite{Lassas2022UniquenessRA}, the second-order linearization leads here to an attenuated ray transform of the difference of the nonlinear potentials. After controlling the damping and linear potential, we derive a stability estimate for this attenuated transform and use it to obtain the H\"older stability estimate for the nonlinear coefficient.

\subsection{Contribution of the present work}

The main contribution of this paper is a stability analysis
for the simultaneous determination of three different types of
time-independent coefficients appearing in the semilinear hyperbolic
equation. The problem requires one to separate a first-order damping term,
a zeroth-order linear potential, and a nonlinear potential from the
nonlinear boundary measurement operator.

A further difficulty is that the coefficients enter at different orders of
the linearization. The damping and linear potential are already present in
the first linearized equation, whereas the coefficient $q$ associated with
the power nonlinearity $u^2$ appears only at second order. We deal with both steps together i.e., the geometric optics solutions address the damping and the linear potential, whereas the asymptotic solutions give the stability of the nonlinear potential.

\subsection{Main ideas of the proof}

We briefly outline the method of proof. The argument relies on a combination of the second-order linearization with suitably constructed geometric optics and asymptotic solutions. The first-order linearization reduces the nonlinear inverse problem to an inverse problem for a linear damped wave equation given as follows
\[
    \Box v+a(x)\partial_tv+b(x)v=0.
\]

Comparing the first linearized equations for the two sets of
coefficients $(a_j, b_j, q_j)$, $j = 1, 2$ yields an integral identity
involving $a_1-a_2$ and $b_1-b_2$. We then choose geometric optics solutions in such a way that the large parameter $\tau$ multiplies the term with $a_1-a_2$, while the term with $b_1-b_2$ contains no such factor. So for large $\tau$, the damping term dominates the $b_1-b_2$ term. This yields a stability estimate for weighted line integrals of the damping
difference and, after localization and ray transform estimates, establishes the H\"older stability for the damping coefficient $a$.

Once we have an estimate for $a_1-a_2$, we return to the same
first-order identity. This time, we move the damping term to the error side,
i.e., the right-hand side in this case, and bound it using the estimate just
obtained. What is left is the term with $b_1-b_2$, and this gives the
stability estimate for the linear potential.

The nonlinear coefficient $q$ first appears in the second order
linearization. The resulting equation contains a source term involving $q$ and products of first-order solutions. Into this equation, we insert asymptotic solutions in which the two forward solutions carry the phase
$e^{i\tau(t+x\cdot\omega)}$ and the backward solution carries
$e^{-2i\tau(t+x\cdot\omega)}$, so that the phases add up to zero and the
leading term of the identity is independent of $\tau$. This reduces the main term in the resulting integral identity to an attenuated ray transform of $q_1-q_2$ along straight lines in $\Omega$.

Using the stability estimates already obtained for $a_1-a_2$ and
$b_1-b_2$, together with the bounds for the remainders in
the asymptotic expansions, we control all lower-order error terms. We then
derive a stability estimate for the corresponding attenuated ray transform
of $q_1-q_2$. Finally, a localization argument and the injectivity and
stability properties of the attenuated ray transform yield the desired
H\"older-type estimate for $q_1-q_2$. Thus the proof follows the hierarchy
\[
a \quad\longrightarrow\quad b \quad\longrightarrow\quad q,
\]
and the H\"older exponent becomes smaller at each step, so the stability
estimate gets weaker as we move from $a$ to $b$ to $q$.

\subsection{Previous results and related literature}
The inverse problem considered in this paper belongs to the broad class of
Calder\'on-type inverse problems, where unknown coefficients of a partial
differential equation are determined from boundary measurements. The subject 
  has become an active area of research since  the pioneering work of Calder\'on~\cite{Calderon1980} on the
recovery of electrical conductivity from boundary voltage and current data.
Over the years, related coefficient identification problems have been extensively
studied for elliptic, parabolic, and hyperbolic equations.

For hyperbolic equations, Bukhgeim and Klibanov~\cite{Bukhgeim1981} studied
the determination of a time-independent coefficient from boundary measurements. The construction of complex geometric optics solutions,
introduced by Sylvester and Uhlmann~\cite{Sylvester1987AGU} for the
Calder\'on problem, subsequently became an important tool in inverse
problems for wave equations. In the inverse scattering setting, a similar approach was developed in \cite{Novikov1988}. Using related oscillatory constructions,
Rakesh and Symes~\cite{Rakesh01011988} established uniqueness for the
recovery of a time-independent potential, while Rakesh~\cite{rakesh1990reconstruction}
derived a reconstruction procedure from boundary data. The simultaneous
determination of damping and potential coefficients in linear wave equation was investigated by
Isakov~\cite{Isakov1991AnIH}. In more general geometric settings, inverse
problems for wave equations with time-independent coefficients were studied
by Eskin~\cite{MR2235639,MR2441006}; see also
\cite[Chapter~8]{isakov2006inverse} for further results on inverse problems
for hyperbolic equations.

Stability for such coefficient recovery problems has also been widely studied. Sun~\cite{MR1059582} obtained
H\"older-type stability for a potential in a wave equation from boundary
measurements. Isakov and Sun~\cite{MR1158175} studied simultaneous stability
for first- and zeroth-order coefficients, whereas Cipolatti and
Lopez~\cite{MR2132903} established Lipschitz- and H\"older-type stability
estimates under different regularity assumptions. Stability for magnetic
perturbations was investigated in~\cite{MR2401822}, while partial boundary
measurements were considered in~\cite{MR2523687}. Related stability results
in geometric settings can be found in
\cite{MR1612709,MR2852371,MR4366889,MR3995367,kumar2025h} and the references therein.

Inverse problems involving time-dependent coefficients are generally more
delicate, since finite-time boundary measurements may not determine an
arbitrary coefficient throughout the entire space-time domain without
additional assumptions. Different approaches have been developed to address
this issue, including infinite-time measurements~\cite{RS91,Sal13},
input-to-output maps~\cite{Isa91}, and recovery in geometrically accessible
regions~\cite{RR91}. Further uniqueness and stability results for
time-dependent coefficients in wave equations can be found in
\cite{Kia16a,Kian2017,MR4124641,MR4343270,Sal14,MR3540317,
MR3595191,MR4013301,MR4191617}.

For nonlinear PDEs, a central technique is the \emph{higher-order
linearization method}, whose use in inverse problems goes back to
Isakov~\cite{Isakov1993OnUI} in the context of nonlinear parabolic
equations. The main idea is to differentiate the nonlinear measurement
operator with respect to small boundary inputs. While the first
linearization yields a linearized equation, higher-order derivatives
produce interaction terms that contain information about the nonlinear
coefficients.

For nonlinear wave equations in Euclidean domains, Nakamura and
collaborators~\cite{NakamuraWatanabe2008,NakamuraWatanabeKaltenbacher2009,
Nakamura2020InverseIB} established uniqueness and reconstruction results
for linear and nonlinear coefficients from boundary measurements.
Lin et al.~\cite{lin2024determining} studied a time-dependent semilinear
wave equation without damping and obtained uniqueness results for nonlinear
coefficients and unknown source terms. Lassas et al.~\cite{Lassas2022UniquenessRA}
studied uniqueness, reconstruction, and stability for a semilinear wave
equation without damping and linear potential terms, and also investigated
the corresponding inverse problem numerically, see \cite{ Lassas2024Numerical}. Related simultaneous recovery
problems involving several nonlinear coefficients and source terms were
considered in~\cite{Qiu2027Uniqueness}. More recently, Bhardwaj et al.~\cite{bhardwaj2026reconstructionpotentialdampingcoefficients}
studied the same semilinear wave model considered in the present work and
established reconstruction results for the associated coefficients. In a
related direction, Kumar et al.~\cite{kumar2026partialdatacoefficientidentification}
investigated uniqueness for the corresponding inverse problem in a partial
data setting.

In a geometric setting, Kurylev, Lassas, and
Uhlmann~\cite{Kurylev2018InversePF} demonstrated that nonlinear wave
interactions can reveal geometric information that is not directly
accessible through the corresponding linear problem. This approach was
further developed by Hintz and Uhlmann~\cite{Hintz2021TheDM} for the
recovery of Lorentzian geometry and nonlinear coefficients. Global
uniqueness results for more general nonlinear real principal-type
equations were obtained in~\cite{Oksanen2020InversePF}.

More recently, inverse problems for physically motivated nonlinear
hyperbolic models, including the Westervelt equation, nonlinear elastic
waves, and related nonlinear wave equations, have attracted substantial
interest; see, for example,
\cite{Uhlmann2023DeterminationOT,DEHOOP2019347,Uhlmann2021NonlinearUI,
Fu2023InversePO,Li2023InversePF,Uhlmann2022AnIB,CLOP,FIKO,LUW,liu2025partialdatainverseproblem}.
 Higher-order linearization has also been applied successfully to several
other nonlinear inverse problems; see
\cite{CARSTEA2019121,Lai2024PartialDI,Kian2020PartialDI,CFKKU,
Choulli2021,HarrachLin2023Simultaneous,KrupchykMaSahooSaloStAmant2025,
LinLiuZhang2022,LassasOksanenSahooSaloTetlow2025,KumarLiuVashisth2026,kumar2026reconstructiontimedependentcoefficientssemilinear}
and the references therein.

The present work builds on the reconstruction framework developed in
\cite{bhardwaj2026reconstructionpotentialdampingcoefficients}, where reconstruction of the coefficients were
established. The stability analysis carried out here requires several
additional ingredients. In particular, we derive finite-difference estimates
associated with the first- and second-order linearizations of the nonlinear
DN map, analyze the dependence of the geometric optics and
asymptotic solutions on the large parameter, obtain localization estimates for
the corresponding ray transforms, and optimize the auxiliary parameters in
order to derive explicit H\"older-type stability estimates.

\subsection{Organization of the paper}
The remainder of the paper is organized as follows.
In Section~\ref{sec:forward-problem}, we introduce the functional framework,
recall the preliminary results required for the subsequent analysis, establish
the well-posedness of the IBVP~\eqref{equation; IBVP}, and construct the
geometric optics and  asymptotic solutions that play a crucial role
in the proof of the main result. Section~\ref{sec:identities} is devoted to the
derivation of the relevant integral identities and the stability analysis of
the damping coefficient. In Section~\ref{sec:linear-potential}, we establish
the stability estimate for the linear potential. Finally, in
Section~\ref{sec:nonlinear-potential}, we use the second-order linearization
argument and establish the stability estimate for the nonlinear potential,
thereby completing the proof of the main result.
\section{Preliminaries}
\label{sec:forward-problem}

In this section, we introduce the functional framework and state the
preliminary results required for the stability analysis. We first recall the
small-data well-posedness of the semilinear IBVP from \cite{bhardwaj2026reconstructionpotentialdampingcoefficients} and
then present the geometric optics and higher-order asymptotic solutions that
will be used in the subsequent sections.
\subsection{Function spaces}
We use  standard notation for the Bochner and vector-valued Sobolev spaces;
see \cite[Chapter~5]{evans2022partial}. Here we introduce only those spaces that are specific to the present problem.

For an integer $m\geq0$, we define the energy space
\begin{equation}\label{eq:energy-space}
\mathcal{E}_m
:=
\bigcap_{k=0}^m
C^k\bigl([0,T];H^{m-k}(\Omega)\bigr)
\end{equation}
which is equipped with the norm
\begin{equation}\label{norm-Em}
\|u\|_{\mathcal{E}_m}^2
:=
\sup_{0\leq t\leq T}
\sum_{k=0}^m
\|\partial_t^ku(t,\cdot)\|_{H^{m-k}(\Omega)}^2.
\end{equation}
With this norm, $\mathcal{E}_m$ is a Banach space. Moreover, for
$m>n+1$, the Sobolev embedding theorem implies the product estimate
\begin{equation}\label{eq:Banach_algebra}
\|\Phi\Psi\|_{\mathcal{E}_m}
\leq
C_m
\|\Phi\|_{\mathcal{E}_m}
\|\Psi\|_{\mathcal{E}_m}.
\end{equation}
Thus $\mathcal{E}_m$ has the Banach algebra property, which is
needed to control the nonlinear term in \eqref{equation; IBVP}. We next introduce the class of boundary data compatible with the
vanishing initial conditions in \eqref{equation; IBVP}. Define
\begin{equation}\label{eq:trace-class}
\mathcal{K}_{m+1}
:=
\left\{
f\in H^{m+1}(\Sigma): \p^{k}_{t}f(0,\cdot)=0\ \ \text{ on }\p\Omega \text{ for }
 k=0,\ldots,m
\right\}.
\end{equation}
For $\delta>0$, we then set
\begin{equation}\label{E-delta}
\mathcal{D}_{m+1}^{\delta}
:=
\left\{
f\in\mathcal{K}_{m+1}:
\|f\|_{H^{m+1}(\Sigma)}<\delta
\right\}.
\end{equation}

\subsection{Small-data well-posedness}
With the above notation and function spaces in place, we are now in a position to state the well-posedness result for the IBVP \eqref{equation; IBVP}. More precisely, for sufficiently small Dirichlet boundary data $f$, the problem \eqref{equation; IBVP} admits a unique solution $u$ which belongs to the open ball
\begin{align*}
    \mathbb{B}_R(0) := \left\{ u \in \mathcal{E}_{m+1} \;:\; \|u\|_{\mathcal{E}_{m+1}} <R\right\},
\end{align*}
for some $R>0$.  Consequently, the associated DN map \eqref{eq:DN-map} is well defined for sufficiently small boundary data. The precise statement is given by the following theorem. 
\begin{theorem}[{\cite[Theorem~2.3]{bhardwaj2026reconstructionpotentialdampingcoefficients}}]
\label{main_thm:welposedness}
Let $m$ be an integer with $m>n+1$ and $(a,b,q)\in\mathcal{A}(M)$. Then there exist constants
$\delta>0$ and $R>0$ such that, for every
$f\in\mathcal D_{m+1}^{\delta}$, the IBVP
\eqref{equation; IBVP} admits a unique solution
$u\in \mathbb{B}_R(0) $. Moreover, the estimate 
\begin{align}
  \|u\|_{\mathcal{E}_{m+1}} + \|\partial_\nu u\|_{H^{m}(\Sigma)}
\le C\,e^{CT} \|f\|_{H^{m+1}(\Sigma)}, 
\end{align}
holds for some constant $C>0$, depending only on $\Omega$ and $M$. 
\end{theorem}

\begin{remark}
The proof of Theorem~\ref{main_thm:welposedness} is based on the
corresponding linear energy estimates, the algebra property \eqref{eq:Banach_algebra} and a contraction argument in
$\mathcal E_{m+1}$.
\end{remark}

 \subsection{Geometric optics and higher-order asymptotic solutions} In this subsection, we introduce the geometric optics and asymptotic solutions
that will be used in the subsequent analysis. We first recall the geometric
optics solutions associated with the linearized IBVP
\begin{align}\label{eqn;v}
\begin{cases}
\mathcal{L}_{a,b}(v):=\Box v(t,x) + a(x)\partial_{t}v(t,x) +b(x)v(t,x)=0,  & (t,x)\in\Omega_{T},\\
v(t,x)=f(t,x),&(t,x)\in\Sigma,\\
v(0,x)=\partial_{t}v(0,x)=0,&x\in\Omega
\end{cases}
\end{align}
 and its associated backward wave problem posed with the second coefficient
pair $(a_2,b_2)$,
 \begin{align}\label{eq;:backward}
   \begin{cases}
      \mathcal{L}^{\ast}_{a_2,b_2} v_{0}(t,x):= \Box v_{0}(t,x)-a_{2}(x)\partial_{t} v_{0}(t,x)+b_{2}(x)v_{0}(t,x)=0, &(t,x)\in\Omega_{T},\\
       v_{0}(T,x)=\partial_{t}v_{0}(T,x)=0,&x\in\Omega.
   \end{cases} 
\end{align}
These solutions are the main tools in the stability analysis of
the damping coefficient $a$ and the linear potential $b$, and the resulting
estimates will then be used in the recovery of the nonlinear potential $q$. The constructions in Lemmas
\ref{lemma: geom opt sol v_{1}} and \ref{lemma: geom opt sol v} follow
standard geometric optics arguments; we refer to
\cite{Isakov1991AnIH} for proof of these lammas.

\begin{lemma}[{\cite[Lemma~~2]{Isakov1991AnIH}}]\label{lemma: geom opt sol v_{1}}
   Let $(a,b,q)\in\mathcal A(M)$, $T>0$ and $\varphi\in C_{c}^{\infty}(\R^{n})$ be such that $\supp(\varphi) \cap\overline{\Omega}=\emptyset$. Then for a fixed $\omega \in \mathbb{S}^{n-1}$ and for any $\tau>0$, there exists a solution of IBVP \eqref{eqn;v} taking  the following form 
  \begin{align}\label{eq: v_{(1)}}
  \begin{split}
 v(t,x) = \varphi(x+t\omega)A^{+}(t,x)\exp(\mathrm{i}\tau(x\cdot\omega + t)) + R(t,x,\tau),
 \end{split}
\end{align}
where 
\begin{align}\label{eq:def-Aplus}
     A^{+}(t,x)=\exp\left(-\frac{1}{2}\int_0^ta(x+s\omega)\,ds\right),
\end{align}
and $R(t,x,\tau)$ vanishes at the initial time, $R(0,\cdot,\tau) = \partial_{t}R(0,\cdot,\tau) =  0$ in $\Omega$. Moreover $R$ satisfies
\begin{align}\label{eq: r_{(1)}}
  \tau\|R\|_{L^{2}(\Omega_{T})}+\|\partial_{t} R\|_{L^{2}(\Omega_{T})} \leq C \|\varphi\|_{H^{3}(\R^n)},
\end{align}
where $C>0$ depends only on $\Omega$, $T$ and $M$.
\end{lemma}
\begin{lemma}[{\cite[Lemma~~3]{Isakov1991AnIH}}]\label{lemma: geom opt sol v}
   Let $(a_2,b_2,q_2)\in\mathcal A(M)$, $T>0$ and $\varphi\in C_{c}^{\infty}(\R^{n})$ be such that $\left(\supp(\varphi)\pm T\omega\right)\cap\overline{\Omega}=\emptyset$. Then for a fixed    $\omega \in \mathbb{S}^{n-1}$ and for any $\tau>0$, there exists a solution of IBVP \eqref{eq;:backward} taking  the following form 
  \begin{align}\label{eq: v}
  v_{0}(t,x) = \varphi(x+t\omega)A_2^{-}(t,x)\exp(-\mathrm{i}\tau(x\cdot\omega + t)) + R_0(t,x,\tau),
\end{align}
where
\begin{align}\label{eq:def-Aminus}
    A_2^{-}(t,x)=\exp\left(\frac{1}{2}\int_0^ta_{2}(x+s\omega)\,ds\right),
\end{align}
 and $R_0(t,x,\tau)$ vanishes at the final time, $R_0(T,\cdot,\tau) = \partial_{t}R_0(T,\cdot,\tau) =  0$ in $\Omega$, and satisfies
\begin{align}\label{eq: estimate of r_{(2)}}
  \tau\|R_{0}\|_{L^{2}(\Omega_{T})}+\|\partial_{t} R_{0}\|_{L^{2}(\Omega_{T})} \leq C \|\varphi\|_{H^{3}(\R^n)},
\end{align}
where $C>0$ depends only on $\Omega$, $T$ and $M$.
\end{lemma}

We next state  asymptotic solutions carrying an expansion in negative powers of $\tau$, whose remainders decay in $L^\infty(\Omega_T)$ as $\tau \to \infty$. This stronger control will be required for the stable 
recovery of the nonlinear potential $q$. The precise construction of these solutions and the corresponding remainder estimates are stated in Lemmas \ref{Asymptotic solutions for IBVP} and  \ref{Asymptotic solutions}  below. We refer to \cite{bhardwaj2026reconstructionpotentialdampingcoefficients} for proof of these lemmas. 

\begin{lemma}[{\cite[Proposition~3.5]{bhardwaj2026reconstructionpotentialdampingcoefficients}}]\label{Asymptotic solutions for IBVP}
 Let $m>n+1$, $\omega\in\mathbb S^{n-1}$, $T>0$, $(a,b,q) \in \mathcal{A}(M)$ and $\varphi\in C_{c}^{\infty}(\R^n)$ satisfy $\operatorname{supp}\varphi\cap\overline{\Omega}=\emptyset$. Then, for every integer $N\ge1$ and every
$\tau>1$, the IBVP \eqref{eqn;v} admits an asymptotic solution  having the form 
    \begin{equation}\label{eqn: IVP geom opt sol 1}
 v(t,x)= e^{i\tau\,(t+x \cdot\omega)}\left(m_{0}(t,x)+\frac{m_{1}(t,x)}{\tau}+\frac{m_{2}(t,x)}{\tau^2}+\cdots+\frac{m_{N}(t,x)}{\tau^N}\right)+R(t,x,\tau),
    \end{equation}
    where the function $m_0$ can be chosen as
    \begin{equation}\label{eqn:m0-explicit}
        m_0(t,x)= \varphi(x+t\omega)\exp\left(-\frac12\int_0^ta(x+s\omega)\,ds\right)
    \end{equation}
for $(t,x)\in \Omega_{T}$. Further $m_{k}$ can be chosen as 
\begin{equation}\label{eqn:mk-explicit}
m_k(t,x)
=
\frac{i}{2}
\int_0^t
\mathcal L_{a,b}m_{k-1}
\bigl(s,x+(t-s)\omega\bigr)
\exp\left(
-\frac12
\int_s^t
a\bigl(x+(t-r)\omega\bigr)\,dr
\right)
\,ds
\end{equation}
for $(t,x)\in \Omega_{T}$ and $1\leq k \leq N$. The correction term $R(t,x,\tau)$ satisfies $R(0,x,\tau)=\partial_t R(0,x,\tau)=0$ for $x\in \Omega$, together with the estimate
\begin{equation}\label{eq: r__{(2)}}
   \|R\|_{W^{1,\infty}(\Omega_{T})}\leq C\tau^{\,m-N}\|\varphi\|_{H^{m+2+2N}(\mathbb{R}^{n})},
\end{equation}
where $C>0$ is independent of $\tau$ and $\varphi$.
    \end{lemma}

    \begin{lemma}[{\cite[Proposition 3.3]{bhardwaj2026reconstructionpotentialdampingcoefficients}}]\label{Asymptotic solutions}
 Let $m>n+1$, $\omega\in\mathbb S^{n-1}$, $T>0$, $\left(a_2,b_2,q_2\right) \in \mathcal{A}(M)$ and $\varphi\in C_{c}^{\infty}(\R^n)$ satisfy $\bigl(\supp\varphi\pm T\omega\bigr)\cap\overline{\Omega}=\emptyset$.  Then, for every integer $N\ge1$ and every
$\tau>1$, the backward wave problem \eqref{eq;:backward}
 admits an asymptotic solution  having the form 
    \begin{equation}\label{eqn: IBVP geom opt sol 1}
 v_{0}(t,x)= e^{-{2}i\tau\,(t+x \cdot\omega)}\left(\widetilde{m}_{0}(t,x)+\frac{\widetilde{m}_{1}(t,x)}{\tau}+\frac{\widetilde{m}_{2}(t,x)}{\tau^2}+\cdots+\frac{\widetilde{m}_{N}(t,x)}{\tau^N}\right)+R_0(t,x,\tau),
    \end{equation}
    where the function $\widetilde{m}_{0}$ can be chosen as
    \begin{equation}\label{eq:tildem0-explicit}
        \widetilde m_0(t,x) = \varphi(x+t\omega)\exp\left(\frac12\int_0^ta_2(x+s\omega)\,ds\right)
    \end{equation}
for $(t,x)\in \Omega_{T}$. Further, $\widetilde{m}_{k}$ can be chosen as 
    \begin{equation}\label{eq:tildemk-explicit}
\widetilde m_k(t,x)
=
-\frac{1}{4i}
\int_t^T
\mathcal L^{\ast}_{a_2,b_2}\widetilde m_{k-1}
\bigl(s,x+(t-s)\omega\bigr)
\exp\left(
-\frac12
\int_t^s
a_2\bigl(x+(t-r)\omega\bigr)\,dr
\right)
\,ds
\end{equation}
 for $(t,x)\in\Omega_{T}$ and $1\leq k \leq N$. The correction term $R_0(t,x,\tau)$ satisfies $R_0(T,x,\tau)=\partial_t R_0(T,x,\tau)=0$ for $x\in\Omega$, together with the estimate
\begin{equation}\label{eq: r_{(0)}}
 \|R_0\|_{W^{1,\infty}(\Omega_{T})}\leq C\tau^{\,m-N}\|\varphi\|_{H^{m+2+2N}(\mathbb{R}^{n})},
\end{equation}
where constant $C>0$ is independent of $\tau$ and $\varphi$.   
\end{lemma}
\section{Stability of the damping coefficient}\label{sec:identities}
We begin with the stability analysis of the damping coefficient. The main
goal of this section is to derive a stability estimate for the difference
$a^{(12)} := a_1-a_2$ in terms of $\eta$ described in \eqref{eq:def-eta-main}. The argument is based on the geometric optics
solutions constructed in Lemmas \ref{lemma: geom opt sol v_{1}} and \ref{lemma: geom opt sol v}, the associated integral identity, and a
localization procedure for the resulting weighted line integrals. The main
result of this section is stated in the following proposition.
\begin{proposition}[Stability for the damping coefficient]\label{prop-damping stability}
Let $m$ be an integer with $m>n+1$, $M>0$, and assume that
$(a_j,b_j,q_j)\in\mathcal{A}(M)$  for  $j=1,2.$
Suppose, in addition, that
$T>\operatorname{diam}(\Omega).$
Then there exist constant $C>0$, depending only on
$\Omega$, $T$ and $M$, such that, for all  $\eta\geq0$,
\begin{align}\label{eq;:main-stability-a}
\left\|a_{1}- a_{2}\right\|_{L^\infty(\Omega)}
\leq
C\eta^{\mu_{a}},
\end{align}
where the H\"older exponent $\mu_a\in(0,1)$ is given by
\[
\mu_{a} := \frac{2(2m-n-1)}{3(2m+1)(2m+5)(m+3)}.
\]
\end{proposition}
The proof of Proposition~\ref{prop-damping stability} is based on a sequence
of auxiliary results. We first establish an integral identity for the first-order linearized equation, which will subsequently be combined with geometric optics solutions and boundary estimates to obtain the desired stability estimates. We begin with the following lemma.

\begin{lemma}[First order identity]\label{lem:identity1}
Let $a_{j},b_{j}\in C_{c}^{\infty}(\Omega)$ for $j=1,2$, $f\in\mathcal K_{m+1}$ and let $v^{(j)}$ solve the following IBVP
\begin{align}\label{eqn:v}
\begin{cases}
\mathcal{L}_{a_{j},b_{j}}v^{(j)}(t,x)=0,  & (t,x)\in\Omega_{T},\\
v^{(j)}(t,x)=f(t,x),&(t,x)\in\Sigma,\\
v^{(j)}(0,x)=\partial_{t}v^{(j)}(0,x)=0,&x\in\Omega,
\end{cases}
\end{align}
and $v_{0}$ solves the backward problem \eqref{eq;:backward}.
Then, we have the following integral identity
\begin{equation}\label{eq:id1}
\int_{\Omega_T}\Bigl(a^{(12)}(x)\,\partial_t v^{(1)}(t,x)+b^{(12)}(x)\,v^{(1)}(t,x)\Bigr)v_{0}(t,x)\,\mathrm{d} x\,\mathrm{d} t
=\int_{\Sigma}\partial_\nu\bigl( v^{(1)}- v^{(2)}\bigr)(t,x)\,v_0(t,x)\,\mathrm{d} S_x\,\mathrm{d}t .
\end{equation}
\end{lemma}
\begin{proof}
    Observe that $v^{(12)}:=v^{(1)}-v^{(2)}$ solves the following IBVP
\begin{equation}\label{eq:difference}
        \begin{cases}
          \left(\Box+ a_{2}(x)\partial_{t}+b_{2}(x) \right)v^{(12)}(t,x)=-a^{(12)}(x)\partial_{t}v^{(1)}(t,x)-b^{(12)}(x)v^{(1)}(t,x), &(t,x)\in\Omega_{T},\\
          v^{(12)}(t,x)=0, &(t,x)\in \Sigma,\\
          v^{(12)}(0,x)=\partial_{t}v^{(12)}(0,x)=0, &x\in\Omega .
        \end{cases}
\end{equation}
Multiplying the governing equation of the above IBVP \eqref{eq:difference} by $v_{0}$ and integrating over $\Omega_{T}$ gives
    \begin{equation}\label{eq:mult}
    \int_{\Omega_{T}}\left(\Box +a_{2}(x)\partial_{t}+b_{2}(x)\right)v^{(12)}\,v_{0}\,\mathrm dx\,\mathrm dt
    =-\int_{\Omega_{T}}\left(a^{(12)}(x)\partial_{t}v^{(1)}+b^{(12)}(x)v^{(1)}\right)v_{0}\,\mathrm dx\,\mathrm dt.
    \end{equation}
We next simplify the terms on the left-hand side of the preceding identity. Using integration by parts, together with the initial conditions
$v^{(12)}(0,\cdot)=\partial_t v^{(12)}(0,\cdot)=0$ and the final conditions
$v_0(T,\cdot)=\partial_t v_0(T,\cdot)=0$ in $\Omega$, we obtain
\begin{align}\label{eqn:idt_1}
\begin{split}
\int_{\Omega_{T}}\partial_{t}^{2}v^{(12)}\,v_{0}\,\mathrm dx\,\mathrm dt
&=\int_{\Omega}\Big[\partial_{t}v^{(12)}\,v_{0}-v^{(12)}\,\partial_{t}v_{0}\Big]_{t=0}^{t=T}\mathrm dx+\int_{\Omega_{T}}v^{(12)}\,\partial_{t}^{2}v_{0}\,\mathrm dx\,\mathrm dt\\
&=\int_{\Omega_{T}}v^{(12)}\,\partial_{t}^{2}v_{0}\,\mathrm dx\,\mathrm dt.
\end{split}
\end{align}
By Green's identity, together with the fact that $v^{(12)}=0$ on $\Sigma$, we obtain
\begin{align}\label{eqn:idt_2}
\begin{split}
-\int_{\Omega_{T}}\Delta v^{(12)}\,v_{0}\,\mathrm dx\,\mathrm dt
&=-\int_{\Omega_{T}}v^{(12)}\,\Delta v_{0}\,\mathrm dx\,\mathrm dt-\int_{\Sigma}\left(\partial_{\nu}v^{(12)}\,v_{0}-v^{(12)}\,\partial_{\nu}v_{0}\right)\mathrm dS_{x}\,\mathrm dt\\
&=-\int_{\Omega_{T}}v^{(12)}\,\Delta v_{0}\,\mathrm dx\,\mathrm dt-\int_{\Sigma}\partial_{\nu}v^{(12)}\,v_{0}\,\mathrm dS_{x}\,\mathrm dt.
\end{split}
\end{align}
 Next, applying integration by parts with respect to $t$ and using
$v^{(12)}(0,\cdot)=0$ and $v_0(T,\cdot)=0$, we obtain
   \begin{align}\label{eqn:idt_3}
\begin{split}
\int_{\Omega_{T}}a_{2}(x)\,\partial_{t}v^{(12)}\,v_{0}\,\mathrm dx\,\mathrm dt
&=\int_{\Omega}a_{2}(x)\Big[v^{(12)}\,v_{0}\Big]_{t=0}^{t=T}\mathrm dx-\int_{\Omega_{T}}v^{(12)}\,a_{2}(x)\,\partial_{t}v_{0}\,\mathrm dx\,\mathrm dt\\
&=-\int_{\Omega_{T}}v^{(12)}\,a_{2}(x)\,\partial_{t}v_{0}\,\mathrm dx\,\mathrm dt.
\end{split}
\end{align}
Combining the three identities \eqref{eqn:idt_1}, \eqref{eqn:idt_2}, and
\eqref{eqn:idt_3}, the identity \eqref{eq:mult} can be rewritten as
    \begin{align*}
    \int_{\Omega_{T}}v^{(12)}\left(\Box -a_{2}(x)\partial_{t}+b_{2}(x)\right)v_{0}\,\mathrm dx\,\mathrm dt-\int_{\Sigma}\partial_{\nu}v^{(12)}\,v_{0}\,\mathrm dS_{x}\,\mathrm dt
    = -\int_{\Omega_{T}}\left(a^{(12)}(x)\partial_{t}v^{(1)}+b^{(12)}(x)v^{(1)}\right)v_{0}\,\mathrm dx\,\mathrm dt.
    \end{align*}
    Since $v_{0}$ is a solution of the backward problem~\eqref{eq;:backward}, the first integral on the left-hand side is equal to zero. Hence, the desired identity~\eqref{eq:id1} follows, completing the proof of the lemma.
    \end{proof}

Next, we substitute the geometric optics solution from
Lemma~\ref{lemma: geom opt sol v_{1}} with $(a,b)=(a_1,b_1)$, which we denote by $v^{(1)}$, together with the solution $v_0$ from Lemma~\ref{lemma: geom opt sol v}, into the integral identity~\eqref{eq:id1}, where $A_1^{+}$ and $R_1$ denote the amplitude \eqref{eq:def-Aplus} and the remainder corresponding to $a=a_1$. This gives 
\begin{equation} \label{eq:id1-GO}
\begin{aligned}
&\mathrm{i}\tau
\int_{\Omega_T}
a^{(12)}(x)\,
\varphi^2(x+t\omega)
A_1^{+}A^{-}
\,\mathrm dx\,\mathrm dt
+
\int_{\Omega_T}
a^{(12)}(x)\,
\varphi(x+t\omega)A^{-}
\partial_t\!\left(\varphi A_1^{+}\right)
\,\mathrm dx\,\mathrm dt\\
&\qquad+
\int_{\Omega_T}
b^{(12)}(x)\,
\varphi^2(x+t\omega)
A_1^{+}A^{-}
\,\mathrm dx\,\mathrm dt
+\mathfrak E(\tau)
=
\int_{\Sigma}
\partial_\nu\!\left(v^{(1)}-v^{(2)}\right)
v_0\,
\mathrm dS_x\,\mathrm dt ,
\end{aligned}
\end{equation}
where the remainder term $\mathfrak E(\tau)$ is given by
\begin{equation}\label{eq:remainder}
\begin{aligned}
\mathfrak E(\tau)
={}&
\int_{\Omega_T}a^{(12)}(x)
\Bigl[
\partial_t(\varphi A_1^{+})
e^{\mathrm{i}\tau(x\cdot\omega+t)}R_0
+
\mathrm{i}\tau\varphi A_1^{+}
e^{\mathrm{i}\tau(x\cdot\omega+t)}R_0
+\varphi A^{-}
e^{-\mathrm{i}\tau(x\cdot\omega+t)}
\partial_tR_1
+\partial_tR_1\,R_0
\Bigr]
\,\mathrm dx\,\mathrm dt\\
&+
\int_{\Omega_T}b^{(12)}(x)
\Bigl[
\varphi A_1^{+}
e^{\mathrm{i}\tau(x\cdot\omega+t)}R_0
+
\varphi A^{-}
e^{-\mathrm{i}\tau(x\cdot\omega+t)}R_1
+R_1R_0
\Bigr]
\,\mathrm dx\,\mathrm dt ,
\end{aligned}
\end{equation}
and we denote by $\mathfrak X$ the sum of the two lower-order terms on
the left-hand side of \eqref{eq:id1-GO}, given by
\begin{equation}\label{eq:def-X}
\begin{aligned}
\mathfrak X
:=\int_{\Omega_T}
a^{(12)}(x)
\varphi(x+t\omega)A^{-}
\partial_t\!\left(\varphi A_1^{+}\right)
\,\mathrm dx\,\mathrm dt
+
\int_{\Omega_T}
b^{(12)}(x)
\varphi^2(x+t\omega)
A_1^{+}A^{-}
\,\mathrm dx\,\mathrm dt .
\end{aligned}
\end{equation}
The bounds for $\mathcal E(\tau)$ and $\mathfrak X$ are given by the following lemma.
\begin{lemma}\label{lem:remainder-estimates}
Let $m>n+1$, $(a_j,b_j,q_j)\in\mathcal A(M)$ for $j=1,2$. Suppose
$\omega\in\mathbb S^{n-1}$, $\tau\geq1$, and 
$\varphi\in C_c^{\infty}(\R^n)$ satisfy the support assumptions of
Lemmas~\ref{lemma: geom opt sol v_{1}} and~\ref{lemma: geom opt sol v}. Then, $\mathfrak E(\tau)$ and $\mathfrak X$ as in \eqref{eq:remainder} and
\eqref{eq:def-X} satisfy
\begin{equation}\label{eq:E-K}
|\mathfrak E(\tau)|
+
|\mathfrak X|
\leq
C\|\varphi\|_{H^3(\R^n)}^2,
\end{equation}
where $C=C(\Omega,T,M)>0$ is independent of $\tau$ and $\varphi$.
\end{lemma}
\begin{proof}
    The admissibility assumption \eqref{eq:admissible-coefficients}  implies $\|a^{(12)}\|_{L^\infty(\Omega)}+\|b^{(12)}\|_{L^\infty(\Omega)}\leq CM$. Moreover, from the explicit expressions of $A_1^{+}$ and $A^{-}$, we have
\begin{alignat}{2}
\label{eq:go-amplitudes}
\|A_1^{+}\|_{L^{\infty}(\Omega_T)}+\|A^{-}\|_{L^{\infty}(\Omega_T)}
   &\le Ce^{TM/2}, &\qquad
\|\partial_tA_1^{+}\|_{L^{\infty}(\Omega_T)}
   &\le CMe^{TM/2},\\
\label{eq:go-remainders}
\|R_0\|_{L^{2}(\Omega_T)}+\|R_1\|_{L^{2}(\Omega_T)}
   &\le \frac{C}{\tau}\|\varphi\|_{H^{3}(\R^{n})}, &\qquad
\|\partial_tR_1\|_{L^{2}(\Omega_T)}
   &\le C\|\varphi\|_{H^{3}(\R^{n})}.
\end{alignat}
We shall also use
$\|\varphi(\cdot+t\omega)\|_{L^{2}(\Omega_T)}
 +\|\omega\cdot\nabla\varphi(\cdot+t\omega)\|_{L^{2}(\Omega_T)}
 \le C\|\varphi\|_{H^{3}(\R^{n})}$.

Next, we denote $\mathfrak E(\tau):=\sum_{k=1}^{7}\mathfrak E_k(\tau)$, where the terms $\mathfrak E_k(\tau)$ are numbered according to their
order of appearance in~\eqref{eq:remainder}.
We begin with $\mathfrak E_1(\tau)$. By the product rule, we have
\begin{equation}\label{eq:dt-amplitude}
\partial_t\!\left(
\varphi(x+t\omega)A_1^{+}(t,x)
\right)
=
(\omega\cdot\nabla\varphi)(x+t\omega)A_1^{+}(t,x)
+
\varphi(x+t\omega)\partial_tA_1^{+}(t,x).
\end{equation}
Hence, the Cauchy-Schwarz inequality along with \eqref{eq:go-amplitudes} and \eqref{eq:go-remainders} gives
\begin{align}\label{eq:E1-est}
\begin{split}
&|\mathfrak E_1(\tau)|
\leq
\|a^{(12)}\|_{L^\infty(\Omega)}
\left\|\partial_t(\varphi A_1^{+})\right\|_{L^2(\Omega_T)}
\|R_0\|_{L^2(\Omega_T)}\\
&\leq
C\Bigl(
\|A_1^{+}\|_{L^\infty(\Omega_T)}
\|\omega\cdot\nabla\varphi(\cdot+t\omega)\|_{L^2(\Omega_T)}
+
\|\partial_tA_1^{+}\|_{L^\infty(\Omega_T)}
\|\varphi(\cdot+t\omega)\|_{L^2(\Omega_T)}
\Bigr)
\|R_0\|_{L^2(\Omega_T)}\\
&\ \leq\
\frac{C}{\tau}\|\varphi\|_{H^3(\R^n)}^2.
\end{split}
\end{align}
For $\mathfrak E_2(\tau)$, the factor $\tau$ is compensated by the
$\tau^{-1}$ decay of $R_0$. Therefore, we have
\begin{align*}
|\mathfrak E_2(\tau)|
\leq
\tau\|a^{(12)}\|_{L^\infty(\Omega)}
\|A_1^{+}\|_{L^\infty(\Omega_T)}
\|\varphi(\cdot+t\omega)\|_{L^2(\Omega_T)}
\|R_0\|_{L^2(\Omega_T)}
\leq
C\|\varphi\|_{H^3( \R^n)}^2.
\end{align*}
Similarly, using the estimate for $\partial_tR_1$, we obtain
\begin{align*}
|\mathfrak E_3(\tau)|
&\leq
\|a^{(12)}\|_{L^\infty(\Omega)}
\|A^{-}\|_{L^\infty(\Omega_T)}
\|\varphi(\cdot+t\omega)\|_{L^2(\Omega_T)}
\|\partial_tR_1\|_{L^2(\Omega_T)}
 \leq
C\|\varphi\|_{H^3( \R^n)}^2.
\end{align*}
For the fourth term, we have
\begin{align*}
|\mathfrak E_4(\tau)|
\leq
\|a^{(12)}\|_{L^\infty(\Omega)}
\|\partial_tR_1\|_{L^2(\Omega_T)}
\|R_0\|_{L^2(\Omega_T)}
\leq
\frac{C}{\tau}
\|\varphi\|_{H^3(\mathbb R^n)}^2.
\end{align*}
We now consider the terms involving $b^{(12)}$. Proceeding in the
same manner, we obtain
\begin{align*}
|\mathfrak E_5(\tau)|
+
|\mathfrak E_6(\tau)|
\leq
\frac{C}{\tau}
\|\varphi\|_{H^3(\mathbb R^n)}^2.
\end{align*}
Finally, 
\begin{align}\label{eq:E7-est}
|\mathfrak E_7(\tau)|
\leq
\|b^{(12)}\|_{L^\infty(\Omega)}
\|R_1\|_{L^2(\Omega_T)}
\|R_0\|_{L^2(\Omega_T)}
\leq
\frac{C}{\tau^2}
\|\varphi\|_{H^3(\mathbb R^n)}^2.
\end{align}
From the estimates \eqref{eq:E1-est}--\eqref{eq:E7-est}, we conclude that
\begin{align}\label{eq:E-total}
|\mathfrak E(\tau)|
\leq
C
\left(
1+\frac{1}{\tau}
\right)
\|\varphi\|_{H^3(\mathbb R^n)}^2
\leq
C\|\varphi\|_{H^3(\mathbb R^n)}^2
\qquad \text{ for }\tau\geq1.
\end{align}
It remains to estimate $\mathfrak X$. Applying the Cauchy-Schwarz inequality and using \eqref{eq:admissible-coefficients}, \eqref{eq:go-amplitudes} and
\eqref{eq:dt-amplitude}, we obtain
\begin{align}\label{eq:X-estimate}
|\mathfrak X|
\leq
C\|\varphi\|_{H^3(\mathbb R^n)}^2.
\end{align}
Combining \eqref{eq:E-total} and \eqref{eq:X-estimate} gives \eqref{eq:E-K},
which completes the proof of the lemma.
\end{proof}
It remains to estimate the boundary term appearing on the right-hand
side of equation \eqref{eq:id1-GO}. For this purpose, we first establish the
following auxiliary lemma.
\begin{lemma}\label{lem:second-order-decomposition}
Let $m > n+1$, $(a_j,b_j,q_j)\in\mathcal{A}(M)$ for $j=1,2$ and $f\in \mathcal{K}_{m+1}$. Consider the semilinear IBVP
\begin{align}\label{eq:semilinear-epsilon}
\begin{cases}
\Box u^{(j)}_{\varepsilon f}+a_j(x)\partial_t u^{(j)}_{\varepsilon f}+b_j(x)u^{(j)}_{\varepsilon f}+q_j(x)\big(u^{(j)}_{\varepsilon f}\big)^2=0, & \text{in }\Omega_T,\\
u^{(j)}_{\varepsilon f}=\varepsilon f, & \text{on }\Sigma,\\
u^{(j)}_{\varepsilon f}(0,x)=\partial_tu^{(j)}_{\varepsilon f}(0,x)=0, & \text{in }\Omega.
\end{cases}
\end{align}
Let $v^{(j)}$ solve the linearized problem \eqref{eqn:v} and  $w^{(j)}$ solve
\begin{align}\label{eqn:w}
\begin{cases}
\Box w^{(j)}+a_j(x)\partial_t w^{(j)}+b_j(x)w^{(j)}=-q_j(x)\big(v^{(j)}\big)^2, & \text{in }\Omega_T,\\
w^{(j)}=0, & \text{on }\Sigma,\\
w^{(j)}(0,x)=\partial_tw^{(j)}(0,x)=0, & \text{in }\Omega.
\end{cases}
\end{align}
Then there exists a constant $C>0$, depending only on $\Omega$, $T$ and $M$, such that for every $\varepsilon\in\mathbb R$ with $|\varepsilon|\,\|f\|_{H^{m+1}(\Sigma)}\leq\delta$, the solution $u^{(j)}_{\varepsilon f}$ admits the decomposition 
\begin{equation}\label{eq:decomposition}
u^{(j)}_{\varepsilon f}=\varepsilon v^{(j)}+\varepsilon^2w^{(j)}+\mathcal R^{(j)}_{\varepsilon}.
\end{equation}
Moreover, the following estimate holds
\begin{align}\label{Decomposition_estimate}
\|\mathcal R^{(j)}_{\varepsilon}\|_{\mathcal{E}_{m+1}}+\|\partial_\nu\mathcal R^{(j)}_{\varepsilon}\|_{H^m(\Sigma)}\leq C|\varepsilon|^3\|f\|_{H^{m+1}(\Sigma)}^3.
\end{align}
\end{lemma}

\begin{proof}
 Define $\mathcal R^{(j)}_{\varepsilon}=u^{(j)}_{\varepsilon f}-\varepsilon v^{(j)}-\varepsilon^2w^{(j)}$. Then, using \eqref{eqn:v}, \eqref{eq:semilinear-epsilon}, and \eqref{eqn:w}, it satisfies
\begin{align}\label{eqn:R}
\begin{cases}
\mathcal{L}_{a_{j},b_{j}}\mathcal R^{(j)}_{\varepsilon}=-q_j\left[\big(u^{(j)}_{\varepsilon f}\big)^2-\varepsilon^2\big(v^{(j)}\big)^2\right], & \text{in }\Omega_T,\\
\mathcal R^{(j)}_{\varepsilon}=0, & \text{on }\Sigma,\\
\mathcal R^{(j)}_{\varepsilon}(0,x)=\partial_t\mathcal R^{(j)}_{\varepsilon}(0,x)=0, & \text{in }\Omega.
\end{cases}
\end{align}
From the well-posedness estimate \cite[Lemma~2.2]{bhardwaj2026reconstructionpotentialdampingcoefficients},
the factorization
$\bigl(u^{(j)}_{\varepsilon f}\bigr)^{2}-\varepsilon^{2}\bigl(v^{(j)}\bigr)^{2}
=\bigl(u^{(j)}_{\varepsilon f}-\varepsilon v^{(j)}\bigr)\bigl(u^{(j)}_{\varepsilon f}+\varepsilon v^{(j)}\bigr)$
and the Banach algebra property \eqref{eq:Banach_algebra}, we have the following estimate
\begin{equation}\label{eq:estimate-R-epsilon}
\begin{aligned}
\bigl\|\mathcal R^{(j)}_{\varepsilon}\bigr\|_{\mathcal E_{m+1}}
+\bigl\|\partial_\nu\mathcal R^{(j)}_{\varepsilon}\bigr\|_{H^{m}(\Sigma)}
&\leq C\Bigl\|q_j\bigl(u^{(j)}_{\varepsilon f}-\varepsilon v^{(j)}\bigr)
\bigl(u^{(j)}_{\varepsilon f}+\varepsilon v^{(j)}\bigr)\Bigr\|_{\mathcal E_{m}}\\[2pt]
&\leq C\|q_j\|_{H^{m}(\Omega)}
\bigl\|u^{(j)}_{\varepsilon f}-\varepsilon v^{(j)}\bigr\|_{\mathcal E_{m+1}}
\Bigl(\bigl\|u^{(j)}_{\varepsilon f}\bigr\|_{\mathcal E_{m+1}}
+|\varepsilon|\bigl\|v^{(j)}\bigr\|_{\mathcal E_{m+1}}\Bigr).
\end{aligned}
\end{equation}
By Theorem \ref{main_thm:welposedness}
and \cite[Lemma~2.2]{bhardwaj2026reconstructionpotentialdampingcoefficients},
we have $\|u^{(j)}_{\varepsilon f}\|_{\mathcal E_{m+1}}\leq
C|\varepsilon|\,\|f\|_{H^{m+1}(\Sigma)}$ and $\|v^{(j)}\|_{\mathcal
E_{m+1}}\leq C\|f\|_{H^{m+1}(\Sigma)}$. Substituting these into
\eqref{eq:estimate-R-epsilon}, we obtain
\begin{equation}\label{eq:estimate-R-epsilon-2}
\bigl\|\mathcal R^{(j)}_{\varepsilon}\bigr\|_{\mathcal E_{m+1}}
+\bigl\|\partial_\nu\mathcal R^{(j)}_{\varepsilon}\bigr\|_{H^{m}(\Sigma)}
\leq C\|q_j\|_{H^{m}(\Omega)}
\bigl\|u^{(j)}_{\varepsilon f}-\varepsilon v^{(j)}\bigr\|_{\mathcal E_{m+1}}\,
|\varepsilon|\,\|f\|_{H^{m+1}(\Sigma)} .
\end{equation}
To handle the term $\bigl\|u^{(j)}_{\varepsilon f}-\varepsilon v^{(j)}\bigr\|_{\mathcal E_{m+1}}$,
observe that $u^{(j)}_{\varepsilon f}-\varepsilon v^{(j)}$ solves
\begin{equation}\label{eqn:difference}
\begin{cases}
\mathcal L_{a_j,b_j}\bigl(u^{(j)}_{\varepsilon f}-\varepsilon v^{(j)}\bigr)
=-q_j(x)\bigl(u^{(j)}_{\varepsilon f}\bigr)^{2}, &\text{in }\Omega_T,\\[2pt]
u^{(j)}_{\varepsilon f}-\varepsilon v^{(j)}=0, &\text{on }\Sigma,\\[2pt]
\bigl(u^{(j)}_{\varepsilon f}-\varepsilon v^{(j)}\bigr)(0,x)
=\partial_t\bigl(u^{(j)}_{\varepsilon f}-\varepsilon v^{(j)}\bigr)(0,x)=0, &\text{in }\Omega .
\end{cases}
\end{equation}
Now again using the well-posedness estimate \cite[Lemma~2.2]{bhardwaj2026reconstructionpotentialdampingcoefficients}, we have
\begin{equation}\label{eq:z-linear-estimate}
\begin{aligned}
\bigl\|u^{(j)}_{\varepsilon f}-\varepsilon v^{(j)}\bigr\|_{\mathcal E_{m+1}}
&\leq C\Bigl\|q_j\bigl(u^{(j)}_{\varepsilon f}\bigr)^{2}\Bigr\|_{\mathcal E_{m}}\\[2pt]
&\leq C\|q_j\|_{H^{m}(\Omega)}\bigl\|u^{(j)}_{\varepsilon f}\bigr\|^{2}_{\mathcal E_{m+1}}\\[2pt]
&\leq C\|q_j\|_{H^{m}(\Omega)}|\varepsilon|^{2}\|f\|^{2}_{H^{m+1}(\Sigma)},
\end{aligned}
\end{equation}
where the last two steps follow from the Banach algebra property \eqref{eq:Banach_algebra}
and the well-posedness of $u^{(j)}_{\varepsilon f}$.
Combining \eqref{eq:estimate-R-epsilon-2} and \eqref{eq:z-linear-estimate}, and using
$\|q_j\|_{H^{m}(\Omega)}\leq M$ from \eqref{eq:admissible-coefficients}, we obtain the estimate
\eqref{Decomposition_estimate}, which completes the proof.
\end{proof}

 We next address the boundary term appearing on the right-hand
side of \eqref{eq:id1-GO}. The estimate needed for this term is established
in the following lemma.  
 \begin{lemma}\label{lem:linearized-DN-estimate}
Let $m>n+1$ and  $(a_j,b_j,q_j)\in\mathcal A(M)$ for $j=1,2$, and assume that 
$\eta<\delta^{3}$. Then there exists a constant $C>0$, depending only on $\Omega$, $T$, and $M$, such that
\begin{equation}\label{eq:lemma-linearized-DN}
\left\|
\partial_\nu\bigl(v^{(1)}-v^{(2)}\bigr)
\right\|_{L^2(\Sigma)}
\leq
C\eta^{2/3}
\|f\|_{H^{m+1}(\Sigma)}
\end{equation}
for every $f\in\mathcal{K}_{m+1}$, where $v^{(j)}$, $j=1,2$, are the solutions
of \eqref{eqn:v} with the common Dirichlet data $f$.
\end{lemma}
\begin{proof}
    Taking normal derivatives of equation \eqref{eq:decomposition}, we have
\begin{equation}\label{eq:normal-v-first}
\partial_\nu\bigl(v^{(1)}-v^{(2)}\bigr)
=
\frac{1}{\varepsilon}
\Big[
\partial_\nu\bigl(u_{\varepsilon f}^{(1)}-u_{\varepsilon f}^{(2)}\bigr)
-\varepsilon^2\partial_\nu\bigl(w^{(1)}-w^{(2)}\bigr)
-\partial_\nu\bigl(\mathcal R_\varepsilon^{(1)}
-\mathcal R_\varepsilon^{(2)}\bigr)
\Big]
\quad\text{on }\Sigma .
\end{equation}
Note that  $\varepsilon f\in\mathcal D^{\delta}_{m+1}$ if and only if
$-\varepsilon f\in\mathcal D^{\delta}_{m+1}$. Subtracting the two decompositions \eqref{eq:decomposition} we get $u_{\varepsilon f}^{(j)}-u_{-\varepsilon f}^{(j)}=2\varepsilon v^{(j)}+\mathcal R_\varepsilon^{(j)}-\mathcal R_{-\varepsilon}^{(j)}$. Hence, taking normal derivatives and subtracting the identities for $j=1$ and $j=2$,
\begin{equation}\label{eq:dn-odd}
\begin{split}
\partial_\nu\bigl(v^{(1)}-v^{(2)}\bigr)
&=
\frac{1}{2\varepsilon}
\Bigl[
\bigl(\Lambda_{a_1,b_1,q_1}-\Lambda_{a_2,b_2,q_2}\bigr)(\varepsilon f)
-
\bigl(\Lambda_{a_1,b_1,q_1}-\Lambda_{a_2,b_2,q_2}\bigr)(-\varepsilon f)\\&
\qquad\qquad-
\partial_\nu
\bigl(\mathcal R_\varepsilon^{(1)}-\mathcal R_\varepsilon^{(2)}\bigr)
+
\partial_\nu
\bigl(\mathcal R_{-\varepsilon}^{(1)}-\mathcal R_{-\varepsilon}^{(2)}\bigr)
\Bigr]
\quad\text{on }\Sigma .
\end{split}
\end{equation}
Taking the $L^2(\Sigma)$ norm in the above equation and using the triangle inequality in conjunction with \eqref{eq:def-eta-main}  and \eqref{Decomposition_estimate}, we obtain
\begin{equation}\label{eq:dn-eps}
\left\|
\partial_\nu\bigl(v^{(1)}-v^{(2)}\bigr)
\right\|_{L^2(\Sigma)}
\leq
C\left(
\frac{\eta}{\varepsilon}
+
\varepsilon^2
\|f\|_{H^{m+1}(\Sigma)}^3
\right).
\end{equation}
We now optimize the right-hand side with respect to
$\varepsilon$ by choosing $\varepsilon=\eta^{1/3}/\|f\|_{H^{m+1}(\Sigma)}$,
which is valid since $\eta<\delta^{3}$. This proves the estimate
\eqref{eq:lemma-linearized-DN}, completing the proof of the lemma.
\end{proof}

We are now in a position to derive the weighted estimate for the
difference of the damping coefficients. Indeed, combining the geometric
optics identity~\eqref{eq:id1-GO} with the estimates obtained above for the
remainder term, the lower-order contributions, and the boundary term, we are
left with the term containing $a^{(12)}$. This yields the estimate stated in the following
lemma.
\begin{lemma}\label{lem:weighted-a}
Let $m>n+1$ and $(a_{j},b_{j},q_{j})\in\mathcal{A}(M)$ for $j=1,2$, and assume $\eta< \delta^3$. Then for every $\tau\geq 1$,
$\omega\in\mathbb{S}^{n-1}$ and 
$\varphi\in C^{\infty}_{c}(\mathbb{R}^{n})$ with
$\operatorname{supp}(\varphi)\cap\overline{\Omega}=\emptyset$ and
$\left(\operatorname{supp}(\varphi)\pm T\omega\right)\cap\overline{\Omega}
=\emptyset$,  the estimate
\begin{equation}\label{eq:id1-GO_2}
\left|\int_{\Omega_T}
a^{(12)}(x)\,
\varphi^2(x+t\omega)
A_1^{+}(t,x)A^{-}(t,x)
\,\mathrm dx\,\mathrm dt\right|\leq C\left( \frac{1}{\tau}+ \eta^{2/3} \tau^{m+2} \right)\| \varphi\|^2_{H^{m+2}(\R^{n})}
\end{equation} 
holds, where $C=C(\Omega,T,M)>0$ is independent of $\tau$ and $\varphi$ .
\end{lemma}
\begin{proof}
To derive the desired estimate, we divide the identity~\eqref{eq:id1-GO} by
$\mathrm{i}\tau$ and use the notation introduced in~\eqref{eq:def-X}. Taking absolute values together with the triangle inequality, we obtain
\begin{align}\label{eq:id1-GO-1}
\left|
\int_{\Omega_T}
a^{(12)}(x)\,
\varphi^{2}(x+t\omega)
A_1^{+}A^{-}
\,\mathrm dx\,\mathrm dt
\right|
\leq
\frac{|\mathfrak X|+|\mathfrak E(\tau)|}{\tau}
+
\frac{1}{\tau}
\left|
\int_{\Sigma}
\partial_\nu\!\left(v^{(1)}-v^{(2)}\right)
v_0\,
\mathrm dS_x\,\mathrm dt
\right|.
\end{align}
We first bound the right-hand side of the preceding estimate. On that note, observe that from \eqref{eq:E-K}, we have
\begin{align}\label{eq:first-term-a-final}
\frac{|\mathfrak X|+|\mathfrak E(\tau)|}{\tau}
\leq
\frac{C}{\tau}
\|\varphi\|_{H^3(\mathbb R^n)}^2\leq \frac{C}{\tau}
\|\varphi\|_{H^{m+2}(\mathbb R^n)}^2. 
\end{align}
In the last inequality, we have used  $m+2\geq3$, so that $H^{m+2}(\R^n) \hookrightarrow H^{3}(\R^n)$ and hence
$\|\varphi\|_{H^3(\mathbb R^n)}
\leq
C\|\varphi\|_{H^{m+2}(\mathbb R^n)}$.
For the other term, we use the Cauchy-Schwarz inequality, Lemma 
\ref{lem:linearized-DN-estimate}, and the trace bounds 
\begin{equation}\label{eq:go-traces}
\norm{f}_{H^{m+1}(\Sigma)}\leq C\tau^{m+2}\norm{\varphi}_{H^{m+2}(\R^{n})},
\qquad
\norm{v_0}_{L^{2}(\Sigma)}\leq C\norm{\varphi}_{H^{m+2}(\R^{n})}.
\end{equation}
Combining these estimates, we have
\begin{align}
   \frac{1}{\tau}\left|\int_{\Sigma}\partial_\nu\!\left(v^{(1)}-v^{(2)}\right)\,v_{0}\, \mathrm dS_x\,\mathrm dt\right|
   &\le\frac{1}{\tau}\left\|{\partial_\nu\!\left(v^{(1)}-v^{(2)}\right)}\right\|_{L^{2}(\Sigma)}\norm{v_{0}}_{L^{2}(\Sigma)}\nonumber\\
   &\le\frac{C}{\tau}\,\eta^{2/3}\,
     \norm{f}_{H^{m+1}(\Sigma)}\norm{v_{0}}_{L^{2}(\Sigma)}\nonumber\\
   &\le C\,\eta^{2/3}\,\tau^{m+2}\norm{\varphi}^{2}_{H^{m+2}(\R^{n})}.\label{eq:boundary-a-final}
\end{align}
Finally, inserting the estimate of the right hand side \eqref{eq:first-term-a-final} and
\eqref{eq:boundary-a-final} into \eqref{eq:id1-GO-1} and using $\tau\geq1$, we arrive at
\begin{align*}
\left|
\int_{\Omega_T}
a^{(12)}(x)\,
\varphi^{2}(x+t\omega)
A_1^{+}(t,x)A^{-}(t,x)
\,\mathrm dx\,\mathrm dt
\right|
\leq
C\left(
\frac{1}{\tau}
+
\eta^{2/3}\tau^{m+2}
\right)
\|\varphi\|_{H^{m+2}(\mathbb R^n)}^2.
\end{align*}
This proves \eqref{eq:id1-GO_2} and completes the proof of the lemma.
\end{proof}
The preceding weighted estimate allows us to control the line integrals of
$a^{(12)}$. After eliminating the attenuation factor and applying a suitable
localization argument, we obtain a stability estimate for the X-ray
transform of $a^{(12)}$, stated in the following lemma.
\begin{lemma}\label{lemma:::id1-GO_2}
Let $m>n+1$ and $(a_{j},b_{j},q_{j})\in\mathcal{A}(M)$ for $j=1,2$, and assume that $\eta<  \delta^3$, $T>\mathrm{diam}(\Omega)$. Then, we have the following estimate
\begin{equation}\label{eq:::id1-GO_2}
\left|\int_{\mathbb{R}}
a^{(12)}(z + s \omega)\,
\,\mathrm ds\right|\leq C\,\eta^{\mu_0},\quad \mu_0:=\frac{2}{3(m+3)(2m+5)}  \quad \mbox{for all}\ (z,\omega)\in T\mathbb{S}^{n-1},
\end{equation} 
where $ T\mathbb{S}^{n-1} := \{\, (z,\omega)\in \mathbb{R}^n\times\mathbb{S}^{n-1}  \,:\, z\cdot\omega=0\}$ is the tangent bundle of the unit sphere $\mathbb{S}^{n-1}$ in $\mathbb{R}^n$ and $C=C(\Omega,T,M)>0$.
\end{lemma}
\begin{proof}
We begin by rewriting the estimate~\eqref{eq:id1-GO_2} by introducing the change of variables
$y=x+t\omega$.
Since $a_j\in C_c^\infty(\Omega)$, we extend them by zero
outside $\Omega$ to all of $\R^n$. Applying Fubini's theorem, we  have
\begin{equation}\label{eq:  weighted-a-after-change}
\Bigg|
\int_0^T\int_{\mathbb R^n}
a^{(12)}(y-t\omega)
A_1^+(t,y-t\omega)
A^-(t,y-t\omega)\varphi^{2}(y)
\,\mathrm dy\,\mathrm dt
\Bigg|
\leq
C
\left(
\frac{1}{\tau}
+
\eta^{2/3}\tau^{m+2}
\right)
\|\varphi\|_{H^{m+2}(\mathbb R^n)}^2 .
\end{equation}
Substituting the values of $A_1^+$ and $A^-$ from the \eqref{eq:def-Aplus} and \eqref{eq:def-Aminus}, and using the change of variables
$\rho=t-s$ in the exponent together with the fundamental theorem of calculus, we get
\begin{equation}\label{eq:exponential-averaged}
\left|
\int_{\mathbb R^n}
\left[
\exp\Bigl(
-\frac12
\int_0^T
a^{(12)}(y-\rho\omega)\,\mathrm d\rho
\Bigr)
-1
\right]
\varphi^2(y)\,\mathrm dy
\right|
\leq
C
\left(
\frac1{\tau}
+
\eta^{2/3}\tau^{m+2}
\right)
\|\varphi\|_{H^{m+2}(\mathbb R^n)}^2 .
\end{equation}
We next localize the preceding estimate. Fix
\begin{equation}\label{eq:choice-of-r}
0<r<\min\left\{1,\ \frac{T-\operatorname{diam}(\Omega)}{3}\right\},
\end{equation}
and define
\begin{equation}\label{dom:Omega_epsilon}
\Omega_r
:=
\left\{
x\in\mathbb R^n:
\frac{r}{2}<\operatorname{dist}(x,\Omega)<r
\right\}.
\end{equation}
Let $\chi\in C_c^\infty(B_1(0))$, $0\leq\chi\leq1$ and $\| \chi\|_{L^{2}(\R^n)}=1$.
For $y_0\in\Omega_r$ and $0<h<\frac{r}{4}$, define
\begin{equation}\label{Phi_delta}
\chi_h(y)
:=
h^{-n/2}
\chi\left(\frac{y-y_0}{h}\right).
\end{equation}
Then
$\operatorname{supp}\chi_h\subset B_h(y_0)$ and $\|\chi_h\|_{H^{m+2}(\mathbb R^n)} \leq C h^{-(m+2)}$.
We first check that $\chi_h$ satisfies the support condition in Lemma~\ref{lem:weighted-a}.
Let $y\in B_h(y_0)$ and $w\in\overline\Omega$. Since $y_0\in\Omega_r$ gives
$\operatorname{dist}(y_0,\overline\Omega)>\frac r2$, and $|y-y_0|<h<\frac r4$, we have
\begin{equation}\label{eq:supp-cond-1}
|y-w|\ \geq\ |y_0-w|-|y-y_0|\ >\ \frac r2-\frac r4\ =\ \frac r4\ >\ 0,
\end{equation}
so that $\operatorname{supp}\chi_h\cap\overline\Omega=\emptyset$. For the second
condition, note that $\operatorname{dist}(y_0,\overline\Omega)<r$ implies
$|y-w|\leq h+r+\operatorname{diam}(\Omega)<2r+\operatorname{diam}(\Omega)$, and hence,
by the choice of $r$ in \eqref{eq:choice-of-r},
\begin{equation}\label{eq:supp-cond-2}
|y\pm T\omega-w|\ \geq\ T-|y-w|\ >\ T-2r-\operatorname{diam}(\Omega)
\ >\ \frac{T-\operatorname{diam}(\Omega)}{3}\ >\ 0 .
\end{equation}
Therefore $\bigl(\operatorname{supp}\chi_h\pm T\omega\bigr)\cap\overline\Omega=\emptyset$,
and $\chi_h$ satisfies both support assumptions of Lemma~\ref{lem:weighted-a}.
Choosing $\varphi=\chi_h$ in \eqref{eq:exponential-averaged} and using the previous estimate, we obtain
\begin{align}\label{estimate:a_1-a_2}
\left|
\int_{\mathbb R^n}
\left[
\exp\left(
-\frac12
\int_0^T
a^{(12)}(y-\rho\omega)\,\mathrm d\rho
\right)
-1
\right]
\chi_h^2(y)\,\mathrm dy
\right|\leq
C h^{-2(m+2)}
\left(
\frac1{\tau}
+
\eta^{2/3}\tau^{m+2}
\right).
\end{align}
On the other hand, we observe that
\begin{equation}\label{eq:pointwise-split}
\begin{aligned}
\left|e^{-\frac{1}{2}\int_{0}^{T}a^{(12)}(y_{0}-\rho\omega)
\,\mathrm{d}\rho}-1\right|
&\le\int_{\mathbb{R}^{n}}\chi_{h}^2(y)
\left|e^{-\frac{1}{2}\int_{0}^{T}a^{(12)}(y_{0}-\rho\omega)
\,\mathrm{d}\rho}
-e^{-\frac{1}{2}\int_{0}^{T}a^{(12)}(y-\rho\omega)
\,\mathrm{d}\rho}\right|\,\mathrm{d}y\\
&\qquad+\left|\int_{\mathbb{R}^{n}}\chi^2_{h}(y)
\left(e^{-\frac{1}{2}\int_{0}^{T}a^{(12)}(y-\rho\omega)
\,\mathrm{d}\rho}-1\right)\,\mathrm{d}y\right| ,
\end{aligned}
\end{equation}
where $\int_{\R^n}\chi_h^2(y)\,\mathrm dy=1$ has been used. Since $\|a^{(12)}\|_{L^{\infty}(\Omega)}\le CM$, both exponents
satisfy $\bigl|\frac12\int_0^T a^{(12)}(\cdot-\rho\omega)\,\mathrm d\rho\bigr|\le MT$, so they belong to the bounded interval $[-MT,MT]$, on which the exponential function is Lipschitz. Hence, the first integrand is bounded by
\begin{equation}\label{eq:first-localization-error}
C \left|\int_{0}^{T}a^{(12)}(y_{0}-\rho\omega)\,\mathrm{d}\rho- \int_{0}^{T}a^{(12)}(y-\rho\omega)\,\mathrm{d}\rho  \right|
\leq C T\left\|\nabla a^{(12)}\right\|_{L^{\infty}(\mathbb{R}^{n})}\,|y_0-y|
\le C\,|y_0-y|\le Ch ,
\end{equation}
where the last two steps follow from $\|\nabla a^{(12)}\|_{L^\infty(\R^{n})}\le C\|a^{(12)}\|_{H^m(\R^{n})}\le CM$ by Sobolev embedding ($m>n+1>\frac n2+1$) together with $|y-y_0|\le h$ on $\supp\chi_h$. Substituting  \eqref{estimate:a_1-a_2} and \eqref{eq:first-localization-error} into \eqref{eq:pointwise-split}, we get
\begin{equation}\label{eq:a-exp-diff}
\left|e^{-\frac{1}{2}\int_{0}^{T}a^{(12)}(y_{0}-\rho\omega)\,\mathrm{d}\rho}-1\right|
\le C\left(h+h^{-2(m+2)}\Bigl(\frac{1}{\tau}+\eta^{2/3}\tau^{m+2}\Bigr)\right).
\end{equation}
Next, we show that the left-hand side of \eqref{eq:a-exp-diff} is bounded below by a multiple of $\bigl|\frac{1}{2}\int_{0}^{T}a^{(12)}(y_{0}-\rho\omega)\,\mathrm{d}\rho\bigr|$. Indeed, by the fundamental theorem of calculus, we have
\begin{equation*}
  \left|\exp(s)-1\right|=\left|\int_{0}^{s}\exp(t)\,\mathrm{d}t\right|\geq \exp(-M_{0})|s|\qquad \text{ for }  |s|\leq M_{0}.
\end{equation*}
Since $\|a_{j}\|_{L^{\infty}(\Omega)}\le M$, we chose
$M_{0}:=MT$ and $|s|:=\bigl|\frac{1}{2}\int_{0}^{T}a^{(12)}(y_{0}-\rho\omega)
\,\mathrm{d}\rho\bigr|$. Combining this with \eqref{eq:a-exp-diff} gives
\begin{equation}\label{eq:a-ray-tau-h}
\left|\int_{0}^{T}a^{(12)}(y_{0}-\rho\omega)\,\mathrm{d}\rho\right|
\le 2e^{M_{0}}\left|e^{-\frac{1}{2}\int_{0}^{T}a^{(12)}(y_{0}-\rho\omega)\,\mathrm{d}\rho}-1\right|
\le C\left(h+h^{-2(m+2)}\Bigl(\frac{1}{\tau}+\eta^{2/3}\tau^{m+2}\Bigr)\right).
\end{equation}
This estimate holds for every $y_0\in\Omega_r$,
$\omega\in\mathbb S^{n-1}$, $0<h<\frac{r}{4}$, and  
$\tau\geq 1$. 

We now choose the parameter $h$ in terms of $\tau$ by balancing the
first two terms on the right-hand side of
\eqref{eq:a-ray-tau-h}. Therefore, we choose
$h=\tau^{-\frac{1}{2m+5}}$ and this choice is valid only when $h<\frac{r}{4}$, that is, when $\tau>\left(\frac{4}{r}\right)^{2m+5}$. We therefore set $\tau_{0}:=\max\{1,\,\left(\frac{4}{r}\right)^{2m+5}\}$ and
restrict to $\tau\geq\tau_{0}$.
For this choice,  $h^{-2(m+2)}\tau^{-1}=\tau^{-\frac{1}{2m+5}}$ and $h^{-2(m+2)}\eta^{2/3}\tau^{m+2}=\eta^{2/3}\tau^{m+2+\frac{2(m+2)}{2m+5}}$, and hence, setting
$\alpha_0:=\frac{1}{2m+5}\in(0,1)$,
we deduce from \eqref{eq:a-ray-tau-h} that
\begin{equation}\label{eq:;ray-transform-a-tau}
\left|
\int_0^T
a^{(12)}(y_0-\rho\omega)\,\mathrm d\rho
\right|
\leq
C
\left(
\frac{1}{\tau^{\alpha_0}}
+
\eta^{2/3}
\tau^{m+2+\frac{2(m+2)}{2m+5}}
\right)\qquad \text{ for }\tau\ge\tau_0 .
\end{equation}
Hence, for every $y_0\in\Omega_r$ and
$\omega\in\mathbb S^{n-1}$, the estimate
\eqref{eq:;ray-transform-a-tau} holds for all $\tau\geq \tau_{0}$. Next, we optimize the above estimate in $\tau$ by substituting  $\tau^{-\alpha_0}=\eta^{2/3}\tau^{\beta}$, where $\beta:=m+2+\frac{2(m+2)}{2m+5}=\frac{(m+2)(2m+7)}{2m+5}$. This gives $\tau=\eta^{-\frac{2}{3(m+3)}}$, and this choice is valid as soon as $\eta$ is small enough, i.e.,
$0<\eta\leq\eta_{0}:=\tau_{0}^{-\frac{3(m+3)}{2}}$. With these choices, we arrive at 
\begin{equation}\label{eq:ray-transform-a-eta}
\left|\int_{0}^{T}a^{(12)}(y_{0}-\rho\omega)\,\mathrm d\rho\right|
\leq C\,\eta^{\frac{2}{3(m+3)(2m+5)}},
\end{equation}
for $y_{0}\in\Omega_{r}$, $\omega\in\mathbb S^{n-1}$ and $\eta\leq \eta_{0}$.

Finally, if $\eta\geq\eta_{\star}:=\min\{\delta^3,\eta_0\}$, then
\begin{equation}\label{eq:ray-transform-a-eta_2}
\left|\int_{0}^{T}a^{(12)}(y_{0}-\rho\omega)\,\mathrm d\rho\right|
\leq 2MT
=\frac{2MT}{\eta_{\star}^{\mu_{0}}}\,\eta_{\star}^{\mu_{0}}
\leq\frac{2MT}{\eta_{\star}^{\mu_{0}}}\,\eta^{\mu_{0}},
\end{equation}
where $\mu_{0}:=\dfrac{2}{3(m+3)(2m+5)}$. Now replacing $C$ by $\max\{C,\,2MT\eta_{\star}^{-\mu_{0}}\}$, we conclude that
\eqref{eq:ray-transform-a-eta} holds for every $\eta>0$, every $y_{0}\in\Omega_{r}$
and every $\omega\in\mathbb S^{n-1}$, with a constant $C$ depending only on
$\Omega$, $T$  and $M$. 

Next, we show that the above estimate in fact holds for every $y_{0}\in\mathbb R^{n}$, with the integral over all of $\mathbb R$, with $C$ replaced by $2C$, and the same exponent $\mu_{0}$. Indeed, let $y_{0}\in\Omega_{r}$ and $\omega\in\mathbb S^{n-1}$. If $\rho\geq T$,
then for every $w\in\overline\Omega$ we have
$|y_{0}-\rho\omega-w|\geq\rho-|y_{0}-w|\geq T-\big(r+\operatorname{diam}(\Omega)\big)>2r$, which means that
the point $y_{0}-\rho\omega$ lies outside $\overline\Omega$. Since
$\operatorname{supp}a^{(12)}\subset\Omega$, we conclude that
$a^{(12)}(y_{0}-\rho\omega)=0$ for every $\rho\geq T$. Therefore $\int_{0}^{T}a^{(12)}(y_{0}-\rho\omega)\,\mathrm d\rho
=\int_{0}^{\infty}a^{(12)}(y_{0}-\rho\omega)\,\mathrm d\rho$.  Replacing $\omega$ by $-\omega$ and adding the two half-lines, we get
\begin{align}\label{eq:ray-transform-a-tau}
   \Bigl|\int_{\mathbb{R}}a^{(12)}(y_0-\rho\omega)\,\mathrm d\rho\Bigr|\le 2C\,\eta^{\frac{2}{3(m+3)(2m+5)}}\quad \text{for}\ y_{0}\in \Omega_{r} \ \text{and} \ \omega\in \mathbb{S}^{n-1}. 
\end{align}
Next, we start with the observation that the left hand side of the estimate \eqref{eq:ray-transform-a-tau} remains unchanged if the base point is moved along its own line, since 
\begin{align}
    \int_{\mathbb{R}}a^{(12)}(z+s\omega)\mathrm ds=\int_{\mathbb{R}}a^{(12)}(z+t\omega+s\omega)\mathrm ds=\int_{\mathbb{R}}a^{(12)}(z+(s+t)\omega)\mathrm ds.
\end{align}
 Let $y_{0}\in \R^{n}\setminus\overline{\Omega}$; there are two possibilities. If the line $y_{0}+t\omega$ does not intersect $\Omega_{r}$, then it does not intersect $\overline{\Omega}$. Since $a\in C_{c}^{\infty}(\Omega)$, the integrand vanishes along this line, so the left hand side of the integral equation is zero, therefore the estimate \eqref{eq:ray-transform-a-tau} holds. Suppose now that the line passing through $y_{0}$ intersects $\Omega_{r}$. Then there exists $t\in\mathbb{R}$ such that $y_{0}+t\omega\in\Omega_{r}$. Therefore, by the preceding observation, we conclude that \eqref{eq:ray-transform-a-tau} is valid for $y_{0}\in\mathbb{R}^{n}\setminus\overline{\Omega}$. It remains to consider $y_{0}\in\overline{\Omega}$. In this case, there exists $t\in\mathbb{R}$ such that $y_{0}+t\omega\in \Omega_{r}$. Once again, the earlier observation implies that \eqref{eq:ray-transform-a-tau} holds for such $y_{0}$. Therefore, the estimate \eqref{eq:ray-transform-a-tau} holds for every $y_{0}\in \mathbb{R}^{n}$. This completes the proof.
\end{proof}
We are now ready to prove Proposition~\ref{prop-damping stability}. First, we define the  X-ray transform of $a^{(12)}$ by
\begin{equation}\label{eq:def-Xray}
Xa^{(12)}(z,\omega):=\int_{\mathbb R}a^{(12)}(z+s\omega)\,\mathrm ds,  \quad \mbox{for all}\ (z,\omega)\in T\mathbb{S}^{n-1}.
\end{equation}
Lemma \ref{lemma:::id1-GO_2} allows us to obtain an estimate for the X-ray transform of $a^{(12)}$, from which the desired $L^\infty(\Omega)$ stability follows by an interpolation argument.

\begin{proof}[Proof of Proposition~\ref{prop-damping stability}]
From Lemma~\ref{lemma:::id1-GO_2}, we have the estimate \eqref{eq:::id1-GO_2} which is $\left|Xa^{(12)}(z,\omega)  \right|  \leq C\,\eta^{\mu_0}$ for $(z,\omega)\in T\mathbb{S}^{n-1}$.
Since $\Omega$ is a bounded domain, therefore $\Omega\subset B_{R_\Omega}(0)$ for some $R_\Omega>0$. Since $z$ is
perpendicular to $\omega$, therefore $|z|$ is the distance from the origin to the line $\{z+s\omega\}$ for $s\in \mathbb{R}$. So when
$|z|>R$ the line stays outside $B_{R}(0)$, hence outside $\Omega$, and $Xa^{(12)}(z,\omega)=0$. Therefore, we conclude
\[
  \|Xa^{(12)}\|^{2}_{L^{2}(T\mathbb{S}^{n-1})}
  =\int_{\mathbb{S}^{n-1}}\int_{\{z\in\omega^{\perp}:\,|z|\le R\}}\bigl|Xa^{(12)}(z,\omega)\bigr|^{2}\,\mathrm dz\,\mathrm d\omega
  \le C\,\eta^{2\mu_{0}},
\]
with a constant depending only on $\Omega$ and $R$. Taking square roots,
$\|Xa^{(12)}\|_{L^{2}(T\mathbb{S}^{n-1})}\le C\,\eta^{\mu_{0}}$. From the stability estimate for the ray transform, see \cite{Natterer1986}, we obtain
$\|a^{(12)}\|_{H^{-1/2}(\mathbb{R}^{n})}\le C\,\|Xa^{(12)}\|_{L^{2}(T\mathbb{S}^{n-1})}$ which gives  $\|a^{(12)}\|_{H^{-1/2}(\mathbb{R}^{n})}\le C\,\eta^{\mu_0}$.

Next, we evaluate $\|a^{(12)}\|_{H^{s}(\R^n)}$ for $\frac{n+1}{2}\leq s\leq n+1$ using the interpolation estimate of $\|a^{(12)}\|_{H^{-1/2}(\R^n)}$ and $\|a^{(12)}\|_{H^{m}(\R^n)}$. We denote
$\langle\xi\rangle:=(1+|\xi|^{2})^{1/2}$ and $\|u\|^{2}_{H^{s}(\mathbb{R}^{n})}
=\int_{\mathbb{R}^{n}}\langle\xi\rangle^{2s}|\widehat{u}(\xi)|^{2}\,\mathrm d\xi$. Therefore, we start with
\begin{align*}
  \|a^{(12)}\|^{2}_{H^{s}(\mathbb{R}^{n})}
  &=\int_{\mathbb{R}^{n}}\langle\xi\rangle^{2s}\bigl|\widehat{a^{(12)}}(\xi)\bigr|^{2}\,\mathrm d\xi\\
  &=\int_{\mathbb{R}^{n}}\Bigl(\langle\xi\rangle^{-1}\bigl|\widehat{a^{(12)}}(\xi)\bigr|^{2}\Bigr)^{\frac{2(m-s)}{2m+1}}
     \Bigl(\langle\xi\rangle^{2m}\bigl|\widehat{a^{(12)}}(\xi)\bigr|^{2}\Bigr)^{\frac{2s+1}{2m+1}}\,\mathrm d\xi\\
  &\le\Bigl(\int_{\mathbb{R}^{n}}\langle\xi\rangle^{-1}\bigl|\widehat{a^{(12)}}(\xi)\bigr|^{2}\,\mathrm d\xi\Bigr)^{\frac{2(m-s)}{2m+1}}
      \Bigl(\int_{\mathbb{R}^{n}}\langle\xi\rangle^{2m}\bigl|\widehat{a^{(12)}}(\xi)\bigr|^{2}\,\mathrm d\xi\Bigr)^{\frac{2s+1}{2m+1}}\\
  &=\|a^{(12)}\|^{\frac{4(m-s)}{2m+1}}_{H^{-1/2}(\mathbb{R}^{n})}\;\|a^{(12)}\|^{\frac{4s+2}{2m+1}}_{H^{m}(\mathbb{R}^{n})} ,
\end{align*}
where the third line follows from H\"older's inequality. Taking square roots on both sides gives
\begin{align}
  \|a^{(12)}\|_{H^{s}(\mathbb{R}^{n})}
  &\leq
  \|a^{(12)}\|^{\frac{2(m-s)}{2m+1}}_{H^{-1/2}(\mathbb{R}^{n})} \; \|a^{(12)}\|^{\frac{2s+1}{2m+1}}_{H^{m}(\mathbb{R}^{n})}\\
  &\leq \bigl(C\eta^{\mu_{0}}\bigr)^{\frac{2(m-s)}{2m+1}}(2M)^{\frac{2s+1}{2m+1}}\\
  &\leq C \eta^{\frac{4(m-s)}{3(2m+1)(2m+5)(m+3)}}, 
\end{align}
for $\frac{n+1}{2}\le s\le n+1$. Now by the Sobolev embedding theorem, we have
\begin{align}\label{estimate;final estimate for a}
  \|a^{(12)}\|_{L^{\infty}(\Omega)}\leq \|a^{(12)}\|_{H^{\frac{n+1}{2}}(\mathbb{R}^{n})}
  \leq C\,\eta^{\frac{2(2m-n-1)}{3(2m+1)(2m+5)(m+3)}}.
\end{align}
This completes the proof of the proposition.
\end{proof}

\section{Stability of the linear potential}\label{sec:linear-potential}
We now investigate the stability of the linear potential. The aim is to
estimate $b_1-b_2$ in terms of  $\eta$ described in \eqref{eq:def-eta-main}. The argument combines the geometric optics
solutions and the corresponding integral identity with the stability estimate
for the damping coefficient obtained in the previous section. After
controlling the terms involving $a_1-a_2$, we derive a stability estimate
for the weighted line integrals of $b_1-b_2$, which lead to the desired
stability result. The precise statement is given by the following proposition.

\begin{proposition}[Stability for the linear potential]\label{Prop:b-stability}
Let $m$ be an integer with $m>n+1$, $M>0$ and let $(a_j,b_j,q_j)\in\mathcal A(M)$, $j=1,2$. Suppose, in addition, that
$T>\operatorname{diam}(\Omega).$
Then there exist constant $C>0$, depending only on
$\Omega$, $T$ and $M$, such that, for all  $\eta\geq0$,
\begin{equation}\label{eq:b-Linfty}
\|b_{1}- b_{2}\|_{L^{\infty}(\Omega)}
  \leq C\,\eta^{\mu_b},
\end{equation}
where the H\"older exponent $\mu_b\in(0,1)$ is given by
\begin{align}
     \mu_b:=\frac{(2m-n-1)^{2}}{3(m+3)(2m+1)^{2}(2m+5)^{2}}.
\end{align}
\end{proposition}
The proof of Proposition~\ref{Prop:b-stability} begins with an
estimate for the weighted line integrals of $b^{(12)} := b_{1}- b_{2}$. The contribution of
the damping coefficient is controlled using the stability result established
in the previous section. This leads to the following auxiliary lemma, which
serves as the starting point for the stability analysis of the linear
potential.
\begin{lemma}\label{lem:b-basic}
    Let $n\geq2$ and $\Omega\subset\mathbb{R}^{n}$ be open, connected, and bounded with a smooth boundary. Suppose $T>\operatorname{diam}(\Omega)$,  $M>0$,  $m$ be integers with $m>n+1$ and $(a_{j},b_{j},q_{j})\in\mathcal{A}(M)$ for $j=1,2$,  and assume $\eta<  \delta^3$. Then for every $\tau\geq 1$,
$\omega\in\mathbb{S}^{n-1}$, and 
$\varphi\in C^{\infty}_{c}(\mathbb{R}^{n})$ with
$\operatorname{supp}(\varphi)\cap\overline{\Omega}=\emptyset$ and
$\left(\operatorname{supp}(\varphi)\pm T\omega\right)\cap\overline{\Omega}
=\emptyset$, the following estimate holds
\begin{equation}\label{eq:b-basic}
  \left|\int_{\Omega_T}b^{(12)}(x)\,\varphi^{2}(x+t\omega)A^{+}_{1}(t,x)A^{-}(t,x)\,\mathrm dx\,\mathrm dt\right|
  \le C\Bigl(\frac1\tau+\|a^{(12)}\|_{L^{\infty}(\Omega)}\,\tau+\eta^{2/3}\tau^{m+2}\Bigr)\|\varphi\|^{2}_{H^{m+2}(\mathbb{R}^{n})},
\end{equation}
for $C>0$ depends only on $\Omega$, $T$, and $M$.
\end{lemma}
\begin{proof}
Starting from the integral identity~\eqref{eq:id1-GO}, we isolate the
term containing $b^{(12)}$. This gives
\begin{equation}\label{eq:id1-GO-b}
\begin{aligned}
\int_{\Omega_T}
b^{(12)}(x)
\varphi^2(x+t\omega)
A_1^{+}A^{-}
\,\mathrm dx\,\mathrm dt
={}&
\int_{\Sigma}
\partial_\nu\!\left(v^{(1)}-v^{(2)}\right)
v_0
\,\mathrm dS_x\,\mathrm dt
-\mathrm i\tau
\int_{\Omega_T}
a^{(12)}(x)
\varphi^2(x+t\omega)
A_1^{+}A^{-}
\,\mathrm dx\,\mathrm dt\\
&-\int_{\Omega_T}
a^{(12)}(x)
\varphi(x+t\omega)A^{-}
\partial_t\!\left(\varphi A_1^{+}\right)
\,\mathrm dx\,\mathrm dt
-\mathfrak E(\tau) .
\end{aligned}
\end{equation}
Taking absolute values and applying the triangle inequality, the left-hand side of \ref{eq:b-basic} is at most $I_1+I_2+I_3+\sum_{k=1}^{7}|\mathfrak E_k(\tau)|$, where $I_1,I_2,I_3$ denotes the absolute values of the three integrals on the right of \eqref{eq:id1-GO-b}, in that order, and the remaining contribution is $\mathfrak{E}(\tau)$.

We first estimate the boundary term. Using  Cauchy-Schwarz inequality together with Lemma~\ref{lem:linearized-DN-estimate} and the trace bounds \eqref{eq:go-traces} gives
\begin{equation}\label{eq:b-I1}
I_1
\leq
\left\|
\partial_\nu(v^{(1)}-v^{(2)})
\right\|_{L^2(\Sigma)}
\|v_0\|_{L^2(\Sigma)}
\leq
C\eta^{2/3}
\|f\|_{H^{m+1}(\Sigma)}
\|v_0\|_{L^2(\Sigma)}\le C\eta^{2/3}\tau^{m+2}\|\varphi\|^2_{H^{m+2}(\R^n)}
\end{equation}
Next, using the bounds \eqref{eq:admissible-coefficients} and \eqref{eq:go-amplitudes}  gives us
\begin{equation}\label{eq:b-I2}
I_2
\leq
\tau
\|a^{(12)}\|_{L^\infty(\Omega)}
\|A_1^{+}\|_{L^\infty(\Omega_T)}
\|A^{-}\|_{L^\infty(\Omega_T)}
\int_{\Omega_T}|\varphi(x+t\omega)|^2\,\mathrm dx\,\mathrm dt
\leq
C\tau
\|a^{(12)}\|_{L^\infty(\Omega)}
\|\varphi\|_{H^{m+2}(\mathbb R^n)}^2 .
\end{equation}
For $I_3$ we use \eqref{eq:dt-amplitude} and the Cauchy-Schwarz inequality, which gives
\begin{align}\label{eq:b-I3}
\begin{split}
    I_3&\le\|a^{(12)}\|_{L^\infty(\Omega)}\|A^-\|_{L^\infty(\Omega_{T})}\|\varphi(\cdot+t\omega)\|_{L^2(\Omega_T)}\bigl(\|A_1^+\|_{L^\infty(\Omega_{T})}\|\omega\cdot\nabla\varphi(\cdot+t\omega)\|_{L^2(\Omega_T)}\\
&\qquad\qquad\qquad+\|\partial_tA_1^+\|_{L^\infty(\Omega_T)}\|\varphi(\cdot+t\omega)\|_{L^2(\Omega_T)}\bigr)\\
    &\le C\|a^{(12)}\|_{L^\infty(\Omega)}\|\varphi\|^2_{H^{m+2}(\R^n)}.
\end{split}
\end{align}
Finally, the seven terms $\mathfrak E_k(\tau)$, $1\leq k\leq 7$, were estimated in Section~\ref{sec:identities} (for more details see the proof of Lemma~\ref{lem:remainder-estimates}), this time we will keep the factors $\|a^{(12)}\|_{L^\infty(\Omega)}$ and $\|b^{(12)}\|_{L^\infty(\Omega)}$ instead of bounding them by $2M$, the same computation gives
\begin{equation}\label{eq:b-E-total}
|\mathfrak E(\tau)|
\leq
C
\left(
\frac{1}{\tau}
+
\|a^{(12)}\|_{L^\infty(\Omega)}
\right)
\|\varphi\|_{H^{m+2}(\mathbb R^n)}^2
\le
C
\left(
\frac{1}{\tau}
+
\|a^{(12)}\|_{L^\infty(\Omega)}\tau
\right)
\|\varphi\|_{H^{m+2}(\mathbb R^n)}^2.
\end{equation}
 Combining \eqref{eq:b-I1}, \eqref{eq:b-I2}, \eqref{eq:b-I3}, and \eqref{eq:b-E-total} proves the estimate \eqref{eq:b-basic}, which completes the proof of the lemma.
\end{proof}

We now proceed to the proof of the main result of this section. The argument combines the weighted estimate obtained above with the
stability estimate for the damping coefficient. This leads to a bound for the
X-ray transform of $b^{(12)}$, from which the desired $L^\infty$-stability
follows by interpolation.
\begin{proof}[Proof of Proposition \ref{Prop:b-stability}]
We next remove the factor of $A_1^{+}A^-$ from the left hand side of estimate \eqref{eq:b-basic}. To do so, we start with a change of variables $y=x+t\omega$ and extend $b_j$ by zero to $\R^n$, as we have done in \eqref{eq:  weighted-a-after-change}. Combining this with the triangle inequality, Lemma \eqref{lem:b-basic} and Proposition \ref{prop-damping stability}, we obtain 
\begin{align}\label{estimate; split}
&
\left|
\int_0^T\int_{\mathbb R^n}
b^{(12)}(y-t\omega)
\varphi^2(y)
\,\mathrm dy\,\mathrm dt
\right|
\nonumber\\
&\leq
\left|
\int_0^T\int_{\mathbb R^n}
b^{(12)}(y-t\omega)
\left[
1-
A_1^+(t,y-t\omega)A^-(t,y-t\omega)
\right]
\varphi^2(y)
\,\mathrm dy\,\mathrm dt
\right|
\nonumber\\
&\quad+
C
\left(
\frac{1}{\tau}
+
\eta^{2/3}\tau^{m+2}
+
\tau\eta^{\mu_a}
\right)
\|\varphi\|_{H^{m+2}(\mathbb R^n)}^2.
\end{align}
On the other hand, using the fact that the exponential function is Lipschitz on bounded domain, we have
\begin{equation}\label{eq:b-weight-removal}
\left|
1-
A_1^+(t,y-t\omega)A^-(t,y-t\omega)
\right|
\leq
C
\left|
\int_0^t
a^{(12)}(y-\rho\omega)\,\mathrm d\rho
\right|
\leq
CT\|a^{(12)}\|_{L^\infty(\Omega)}\le C\eta^{\mu_a} .
\end{equation}
Substituting \eqref{eq:b-weight-removal} into \eqref{estimate; split} together with the bound $\|b^{(12)}\|_{L^\infty(\Omega)}\le 2M$, we get
\begin{equation}\label{eq:b-noweight}
\left|
\int_0^T\!\!\int_{\mathbb R^n}
b^{(12)}(y-t\omega)
\varphi^2(y)
\,\mathrm dy\,\mathrm dt
\right|
\leq
C
\left(
\frac{1}{\tau}
+
\eta^{\mu_a}\tau
+
\eta^{2/3}\tau^{m+2}
\right)
\|\varphi\|_{H^{m+2}(\mathbb R^n)}^2.
\end{equation}
Proceed exactly as we did in section \ref{sec:identities} and choosing $\varphi=\chi_{h}$, we arrive at
\begin{equation}\label{eq:b-Hminus-h-tau}
\|b^{(12)}\|_{H^{-1/2}(\mathbb R^n)} \leq
C
\left[
h
+
\frac{h^{-2(m+2)}}{\tau}
+
h^{-2(m+2)}\eta^{\mu_a}\tau
+
h^{-2(m+2)}\eta^{2/3}\tau^{m+2}
\right].
\end{equation}
This estimate holds for  $0<h<r/4$, and  
$\tau\geq \tau_{0}$.

We optimize \eqref{eq:b-Hminus-h-tau} in $h$ and $\tau$. Balancing the first two terms as before, we set $h=\tau^{-\alpha_0}$ to obtain
\begin{equation}\label{eq:b-factorized}
\|b^{(12)}\|_{H^{-1/2}(\mathbb R^n)}
\leq
C\tau^{-\alpha_0}
\left(
1
+
\eta^{\mu_a}\tau^2
+
\eta^{2/3}\tau^{m+3}
\right).
\end{equation}
We now choose $\tau$ so as to balance the first and the second terms in the bracket by choosing $\tau=\eta^{-\mu_a/2}$
which satisfies $\tau\ge\tau_0$ for all $0<\eta\le\eta_1$ with $\eta_1>0$ small enough. It remains to check that the third term stays bounded under this choice. Indeed $\eta^{2/3}\tau^{m+3}=\eta^{\frac23-\frac{\mu_a(m+3)}{2}}$ and
\begin{equation}\label{eq:b-side-condition}
\frac{\mu_a(m+3)}{2}=\frac{2m-n-1}{3(2m+1)(2m+5)}<\frac{2m+1}{3(2m+1)(2m+5)}=\frac{1}{3(2m+5)}<\frac13<\frac23,
\end{equation}
so that $\eta^{2/3}\tau^{m+3}\le1$ for $0<\eta\le1$. Therefore, \eqref{eq:b-factorized} gives $\|b^{(12)}\|_{H^{-1/2}(\R^n)}\le C\tau^{-\alpha_0}=C\eta^{\alpha_0\mu_a/2}$, that is,
\begin{equation}\label{eq:b-Hminus}
\|b^{(12)}\|_{H^{-1/2}(\mathbb R^n)}
\leq
C
\eta^{\nu_b},
\qquad
\nu_b:=\frac{\mu_a}{2(2m+5)}=\frac{2m-n-1}{3(m+3)(2m+1)(2m+5)^{2}} .
\end{equation}
For $\eta>\eta_1$ the estimate holds trivially after adjusting the constant $C$, as in Section~\ref{sec:identities}. Combining \eqref{eq:b-Hminus} with the a priori bound $\|b^{(12)}\|_{H^m(\R^n)}\le2M$, the interpolation inequality gives
\[
\|b^{(12)}\|_{H^{\frac{n+1}{2}}(\R^n)}\le\|b^{(12)}\|^{\frac{2m-n-1}{2m+1}}_{H^{-1/2}(\R^n)}\|b^{(12)}\|^{\frac{n+2}{2m+1}}_{H^{m}(\R^n)}\leq C\eta^{\mu_b} .
\]
Combining the above estimate with the Sobolev embedding $\|b^{(12)}\|_{L^{\infty}(\R^n)}\leq \|b^{(12)}\|_{H^{\frac{n+1}{2}}(\R^n)}$ completes the proof of the Proposition~\ref{Prop:b-stability}.
\end{proof}

\section{Stability for the nonlinear potential.}\label{sec:nonlinear-potential}
We now turn to the stability analysis of the nonlinear potential $q$. In
contrast to the damping coefficient $a$ and the linear potential $b$, which
already appear in the first-order linearized equation, the coefficient $q$
 appears only in the second-order linearization of the nonlinear
DN map. The estimates obtained in the preceding sections
for $a^{(12)}$ and $b^{(12)}$ will therefore play an important role in
controlling the additional terms arising in the present analysis. Throughout this section, we fix $N=m+1$ appearing in Lemmas  \ref{Asymptotic solutions for IBVP} and \ref{Asymptotic solutions}. We also set $\gamma:=\frac{\mu_b}{2}$ and $\kappa:= 9m+12-\frac{n}{2}$.
\medskip

The main result of this section is given by the following proposition.
\begin{proposition}[Stability for the nonlinear potential]\label{prop:q-stability}
  Let $m$ be an integer with $m>n+1$, $M>0$ and  $(a_j,b_j,q_j)\in\mathcal A(M)$ for $j=1,2$. Suppose, in addition, that
$T>\operatorname{diam}(\Omega).$
Then there exist constant $C>0$, depending only on
$\Omega$, $T$ and $M$, such that, for all  $\eta\geq0$,
\begin{equation}\label{eq:q-Linfty}
\|q_{1}- q_{2}\|_{L^{\infty}(\Omega)}\leq C\,\eta^{\mu_q},
\end{equation}
where the H\"older exponent $\mu_q\in(0,1)$ is given by 
\[
\mu_q
:=
\frac{\mu_{b}}
{2(18m+26-n)}.
\]
\end{proposition}

We begin by deriving an integral identity involving the difference
$q^{(12)} := q_{1}- q_{2}$. The precise statement is given by the following lemma.
\begin{lemma}
    Let $(a_j,b_j,q_j)\in\mathcal{A}(M)$ and $v^{(j)}$ solve \eqref{eqn:v} with the same Dirichlet data $f\in\mathcal{K}_{m+1}$, assume $w^{(j)}$ solve \eqref{eqn:w} for $j=1,2$, and $v_{0}$ solve \eqref{eq;:backward}. Then the following identity holds
 \begin{align}\label{eq:q-identity}
 \begin{split}
        \int_{\Omega_{T}}q^{(12)}\big[v^{(1)}\big]^{2}v_{0}\,\mathrm dx\,\mathrm dt
  &=\int_{\Sigma}\big(\partial_{\nu}w^{(1)}-\partial_{\nu}w^{(2)}\big)v_{0}\,\mathrm dS_{x}\,\mathrm dt
  -\int_{\Omega_{T}}a^{(12)}\partial_{t}w^{(1)}v_{0}\,\mathrm dx\,\mathrm dt\\&
  -\int_{\Omega_{T}}b^{(12)}w^{(1)}v_{0}\,\mathrm dx\,\mathrm dt
  -\int_{\Omega_{T}}q_{2}\big(v^{(1)}-v^{(2)}\big)\big(v^{(1)}+v^{(2)}\big)v_{0}\,\mathrm dx\,\mathrm dt,
  \end{split}
    \end{align}
\end{lemma}
\begin{proof}
    Observe that $w^{(1)}-w^{(2)}$ solves the following IBVP
\begin{equation}\label{eq:difference-w}
  \begin{cases}
    \left(\Box+a_{2}\partial_{t}+b_{2}\right)\left(w^{(1)}-w^{(2)}\right)
    =-a^{(12)}\partial_{t}w^{(1)}
    -b^{(12)}w^{(1)}
    -\left(q_{1}\left[v^{(1)}\right]^2-q_{2}\left[v^{(2)}\right]^2\right),
    &\text{in }\Omega_{T},\\[4pt]
    \left(w^{(1)}-w^{(2)}\right)=0, &\text{on }\Sigma,\\[2pt]
    \left(w^{(1)}-w^{(2)}\right)(0,\cdot)=\partial_{t}\left(w^{(1)}-w^{(2)}\right)(0,\cdot)=0, &\text{in }\Omega.
  \end{cases}
\end{equation}
Multiplying the governing equation of the above IBVP \eqref{eq:difference-w} by $v_{0}$,
integrating over $\Omega_{T}$, and proceeding exactly as in the proof of
Lemma~\ref{lem:identity1}, we obtain
\begin{align}\label{eq:integrated-difference}
       \int_{\Omega_{T}}a^{(12)}\partial_{t}w^{(1)}v_{0}\,\mathrm dx\,\mathrm dt
      +\int_{\Omega_{T}}b^{(12)}w^{(1)}v_{0}\,\mathrm dx\,\mathrm dt
      &+\int_{\Omega_{T}}\big(q_{1}[v^{(1)}]^2-q_{2}[v^{(2)}]^2\big)v_{0}\,\mathrm dx\,\mathrm dt
     \\& = \int_{\Sigma}\partial_{\nu}\big(w^{(1)}-w^{(2)}\big)\,v_{0}\,\mathrm dS_{x}\,\mathrm dt.
\end{align}
Finally, combining the above identity with the following algebraic identity
    \[
    q_1[v^{(1)}]^2-q_2[v^{(2)}]^2
  =q^{(12)}[v^{(1)}]^2+q_2\big([v^{(1)}]^2-[v^{(2)}]^2\big)
  =q^{(12)}[v^{(1)}]^2+q_2\big(v^{(1)}-v^{(2)}\big)\big(v^{(1)}+v^{(2)}\big),
    \] 
we get \eqref{eq:q-identity}, which completes the proof of the lemma.
\end{proof}
We next derive an estimate for the boundary term appearing in the preceding
integral identity. This will provide the boundary control needed in the
stability analysis of the nonlinear potential.
\begin{lemma}\label{lem:w-DN}
    Let $m>n+1$, $(a_{j},b_{j},q_{j})\in \mathcal{A}(M)$ for $j=1,2$, and assume $\eta<\delta^{3}$. Then, for every $f\in\mathcal{K}_{m+1}$,
    \begin{equation}\label{q:w-DN}
\bigl\|\partial_\nu\bigl(w^{(1)}-w^{(2)}\bigr)\bigr\|_{L^2(\Sigma)}
\le C\eta^{1/3}\,\|f\|^{2}_{H^{m+1}(\Sigma)} ,
\end{equation}
where $w^{(1)},w^{(2)}$ are the solutions of \eqref{eqn:w} associated
with the same Dirichlet data $f$.
\end{lemma}
\begin{proof}
    From the decomposition \eqref{eq:decomposition}, we get two separate equations for $j=1$ and $j=2$, then subtracting the two equations together with the definition \eqref{eq:DN-map} of the DN maps, we get
    \begin{equation}\label{eqn:w-DN-identity}
\varepsilon^{2}\,\partial_\nu\bigl(w^{(1)}-w^{(2)}\bigr)
=\bigl(\Lambda_{a_1,b_1,q_1}-\Lambda_{a_2,b_2,q_2}\bigr)(\varepsilon f)
-\varepsilon\,\partial_\nu\bigl(v^{(1)}-v^{(2)}\bigr)
-\partial_\nu\bigl(\mathcal R^{(1)}_{\varepsilon}
-\mathcal R^{(2)}_{\varepsilon}\bigr).
\end{equation}
Replacing $\varepsilon f$  by $-\varepsilon f$ and adding the two identities, we obtain
 \begin{equation}\label{q:w-DN-identity}
 \begin{split}
2\varepsilon^{2}\,\partial_\nu\bigl(w^{(1)}-w^{(2)}\bigr)
&=\bigl(\Lambda_{a_1,b_1,q_1}-\Lambda_{a_2,b_2,q_2}\bigr)(\varepsilon f)+ \bigl(\Lambda_{a_1,b_1,q_1}-\Lambda_{a_2,b_2,q_2}\bigr)(-\varepsilon f)\\&\qquad \qquad
-\partial_\nu\bigl(\mathcal R^{(1)}_{\varepsilon}+\mathcal R^{(1)}_{-\varepsilon}
-\mathcal R^{(2)}_{\varepsilon}-\mathcal R^{(2)}_{-\varepsilon}\bigr).
 \end{split}
\end{equation}
Using $\varepsilon f\in \mathcal{D}^{\delta}_{m+1}$
together with \eqref{eq:def-eta-main} and
\eqref{Decomposition_estimate}, we obtain
\begin{equation}\label{q:w-DN-three-terms}
\bigl\|\partial_\nu\bigl(w^{(1)}-w^{(2)}\bigr)\bigr\|_{L^2(\Sigma)}
\le
C\Bigl(\frac{\eta}{\varepsilon^{2}}
+\varepsilon\|f\|_{H^{m+1}(\Sigma)}^{3}\Bigr) .
\end{equation}
Choosing $\varepsilon=\eta^{1/3}/\|f\|_{H^{m+1}(\Sigma)}$ makes the two terms
equal to $\eta^{1/3}\|f\|^2_{H^{m+1}(\Sigma)}$, and the choice is
valid because $|\varepsilon|\,\|f\|_{H^{m+1}(\Sigma)}=\eta^{1/3}<\delta$,
which is the hypothesis $\eta<\delta^3$. This proves \eqref{q:w-DN}, which completes the proof of the lemma.
\end{proof}
We next estimate the terms appearing on the right-hand side of the integral
identity~\eqref{eq:q-identity}. For this purpose, we first collect
several auxiliary bounds for the amplitudes and the corresponding solutions.

From the  construction of $m_{k}$, $\widetilde{m}_k$ in \eqref{eqn:m0-explicit},  \eqref{eqn:mk-explicit}, \eqref{eq:tildem0-explicit}, and \eqref{eq:tildemk-explicit}, we have
\begin{equation}\label{eq:mk-bounds}
\|m_k\|_{\mathcal E_{m+2}}+\|\widetilde m_k\|_{\mathcal E_{m+2}}
\le C\|\varphi\|_{H^{m+2+2k}(\mathbb R^n)},\qquad 0\le k\le N.
\end{equation}
Since $K>3m+4+\frac n2$, the Sobolev embedding
$H^{K}(\Omega)\hookrightarrow C^{3m+4}(\overline{\Omega})$ gives
$\|a\|_{C^{3m+4}(\overline{\Omega})}+\|b\|_{C^{3m+4}(\overline{\Omega})}\le CM$, and each step of the
recursions (2.9) and (2.13) applies the second-order operator $L_{a,b}$ once,
so with $N=m+1$ the constant in (5.4) depends only on $M$. Since $m>n+1$ the Sobolev embedding
$\mathcal E_{m+2}\hookrightarrow L^\infty(\Omega_T)$ implies
\begin{equation}
    \|m_k\|_{L^{\infty}(\Omega_{T})}\leq\|m_k\|_{\mathcal{E}_{m+2}}\leq C\|\varphi\|_{H^{m+2+2k}(\mathbb{R}^{n})},\qquad 0\le k\le N.
\end{equation}
Using the asymptotic representation
\eqref{eqn: IVP geom opt sol 1} and applying the
remainder estimate~\eqref{eq: r__{(2)}}, we obtain, for every $\tau\ge1$,
\begin{align*}
\|v^{(1)}\|_{L^\infty(\Omega_T)}
\le
\left(\sum_{k=0}^{N}\tau^{-k}\right)
\max_{0\le k\le N}
\|m_k\|_{L^\infty(\Omega_T)}
+\|R\|_{L^\infty(\Omega_T)}\leq
C\|\varphi\|_{H^{m+2+2N}(\mathbb R^n)}.
\end{align*}
The same argument applied to the backward asymptotic solution $v_0$ gives
\begin{equation}\label{eq:v-w0-Linfty}
\|v^{(1)}\|_{L^\infty(\Omega_T)}
+\|v_0\|_{L^\infty(\Omega_T)}
\le
C\|\varphi\|_{H^{m+2+2N}(\mathbb R^n)}.
\end{equation}
We next estimate the second-order term $w^{(1)}$. Since $w^{(1)}$ solves
\eqref{eqn:w} with $j=1$, the energy estimate
\cite[Lemma~2.2]{bhardwaj2026reconstructionpotentialdampingcoefficients}
yields
\begin{equation}\label{eq:w1-energy}
\|w^{(1)}\|_{\mathcal E_1}
\le C\bigl\|q_1[v^{(1)}]^2\bigr\|_{L^1(0,T;L^2(\Omega))}
\le C\|q_1\|_{L^\infty(\Omega)}\|v^{(1)}\|^2_{L^\infty(\Omega_T)}
\le C\|\varphi\|_{H^{m+2+2N}(\mathbb{R}^{n})}^{2} .
\end{equation}
We shall also need an estimate for the time derivative of $v^{(1)}$.
Differentiating the expansion~\eqref{eqn: IVP geom opt sol 1}, we obtain
\[
\partial_{t}v^{(1)}= \mathrm{i}\tau e^{i\tau\,(t+x \cdot\omega)}\sum_{k=0}^{N}\tau^{-k}m_{k} +e^{i\tau\,(t+x \cdot\omega)}\sum_{k=0}^{N}\tau^{-k}\partial_{t}m_{k}+ \partial_{t}R.
\]
Combining this identity with 
\eqref{eq: r__{(2)}} and \eqref{eq:mk-bounds}, we deduce
\begin{equation}\label{eq:dt-v1-Linfty}
\|\partial_t v^{(1)}\|_{L^\infty(\Omega_T)}
\le
C\tau
\|\varphi\|_{H^{m+2+2N}(\mathbb R^n)},
\qquad
\tau\ge1.
\end{equation}
Finally, the difference $v^{(1)}-v^{(2)}$ satisfies
\eqref{eq:difference}. Applying the same energy estimate and using
\eqref{eq:v-w0-Linfty} and \eqref{eq:dt-v1-Linfty}, we obtain
\begin{equation}\label{eq:V-energy}
\begin{aligned}
\|v^{(1)}-v^{(2)}\|_{\mathcal E_1}
&\le C\bigl\|a^{(12)}\partial_tv^{(1)}+b^{(12)}v^{(1)}\bigr\|_{L^1(0,T;L^2(\Omega))}\\&
\le C\tau\left( \|a^{(12)}\|_{L^\infty(\Omega)}+\|b^{(12)}\|_{L^\infty(\Omega)} \right)\|\varphi\|_{H^{m+2+2N}(\mathbb{R}^{n})} .
\end{aligned}
\end{equation}
The estimates
\eqref{eq:v-w0-Linfty}--\eqref{eq:V-energy}
provide the bounds needed to control the lower-order terms on the
right-hand side of the integral identity~\eqref{eq:q-identity}.  

We first consider the terms involving the already recovered coefficients
$a^{(12)}$ and $b^{(12)}$. By \eqref{eq:v-w0-Linfty} and
\eqref{eq:w1-energy}, we obtain
\begin{equation}\label{eq:q-rhs-ab}
\begin{aligned}
\left|\int_{\Omega_{T}}a^{(12)}\partial_{t}w^{(1)}v_{0}\,dx\,dt
+\int_{\Omega_{T}}b^{(12)}w^{(1)}v_{0}\,dx\,dt\right|
&\leq C\left(\bigl\|a^{(12)}\bigr\|_{L^{\infty}(\Omega)}
+\bigl\|b^{(12)}\bigr\|_{L^{\infty}(\Omega)}\right)
\bigl\|w^{(1)}\bigr\|_{\mathcal{E}_{1}}\|v_{0}\|_{L^{2}(\Omega_T)}\\
&\leq C\left(\bigl\|a^{(12)}\bigr\|_{L^{\infty}(\Omega)}
+\bigl\|b^{(12)}\bigr\|_{L^{\infty}(\Omega)}\right)
\|\varphi\|^{3}_{H^{m+2+2N}(\mathbb{R}^{n})}.
\end{aligned}
\end{equation}
We next estimate the quadratic term. Using $v^{(1)}+v^{(2)}=2v^{(1)}-(v^{(1)}-v^{(2)})$ together with \eqref{eq:v-w0-Linfty} and \eqref{eq:V-energy}, we get
\begin{equation}\label{eq:q-rhs-q2}
\begin{aligned}
&\left|\int_{\Omega_{T}}q_{2}\big(v^{(1)}-v^{(2)}\big)\big(v^{(1)}+v^{(2)}\big)v_{0}\,dx\,dt\right|\\
&\qquad\leq\left|\int_{\Omega_{T}}q_{2}\big(v^{(1)}-v^{(2)}\big)2v^{(1)}v_{0}\,dx\,dt\right|
+\left|\int_{\Omega_{T}}q_{2}\big(v^{(1)}-v^{(2)}\big)^{2}v_{0}\,dx\,dt\right|\\
&\qquad\leq C\bigl\|v^{(1)}-v^{(2)}\bigr\|_{\mathcal{E}_{1}}
\|v^{(1)}\|_{L^{\infty}(\Omega_{T})}\|v_{0}\|_{L^{2}(\Omega_{T})}
+C\bigl\|v^{(1)}-v^{(2)}\bigr\|^{2}_{\mathcal{E}_{1}}\|v_{0}\|_{L^{\infty}(\Omega_T)}\\
&\qquad\leq C\tau\left(\bigl\|a^{(12)}\bigr\|_{L^{\infty}(\Omega)}
+\bigl\|b^{(12)}\bigr\|_{L^{\infty}(\Omega)}\right)
\left[1+\tau\left\{\bigl\|a^{(12)}\bigr\|_{L^{\infty}(\Omega)}
+\bigl\|b^{(12)}\bigr\|_{L^{\infty}(\Omega)}\right\}\right]
\|\varphi\|^{3}_{H^{m+2+2N}(\mathbb{R}^{n})}.
\end{aligned}
\end{equation}
Finally, for the boundary contribution, Lemma~\ref{lem:w-DN}
together with the trace estimates~\eqref{eq:go-traces} gives
\begin{equation}\label{eq:q-rhs-bdry}
\left|\int_\Sigma\partial_\nu\bigl(w^{(1)}-w^{(2)}\bigr)v_{0}\,\mathrm dS_x\,\mathrm dt\right|
\le C\eta^{1/3}\|f\|^2_{H^{m+1}(\Sigma)}\|v_{0}\|_{L^2(\Sigma)}
\le C\eta^{1/3}\tau^{2m+4}\|\varphi\|^{3}_{H^{m+2+2N}(\mathbb{R}^{n})} .
\end{equation}
We now turn to the principal term on the left-hand side of
\eqref{eq:q-identity}. To simplify the notation, we define
\[
P:=e^{\mathrm i\tau(t+x\cdot\omega)}\sum_{k=0}^N\tau^{-k}m_k
\qquad\text{ and }\qquad
Q:=e^{-2\mathrm i\tau(t+x\cdot\omega)}\sum_{k=0}^N\tau^{-k}\widetilde m_k.
\]
With this notation, the asymptotic solutions
\eqref{eqn: IVP geom opt sol 1} and
\eqref{eqn: IBVP geom opt sol 1} take the form
\[
v^{(1)}=P+R,
\qquad\text{ and }\qquad
v_0=Q+R_0.
\]
The leading contribution in the product
$\bigl(v^{(1)}\bigr)^2v_0$ is given by the amplitudes
$m_0$ and $\widetilde m_0$, whereas all remaining terms contain at least
one negative power of $\tau$ or one remainder term. Combining these
observations with the estimates
\eqref{eq:q-rhs-ab}--\eqref{eq:q-rhs-bdry}
leads to the basic weighted estimate for $q^{(12)}$ stated below.
\begin{lemma}\label{lem:q-basic}
 Let $n\ge2$ and let $\Omega\subset\mathbb R^n$ be open, connected,
and bounded with smooth boundary. Suppose that
$T>\operatorname{diam}(\Omega)$, $M>0$, $m>n+1$, and
$(a_j,b_j,q_j)\in\mathcal A(M)$ for $j=1,2$. Assume also that
$\eta<\delta^3$. Then there exists a constant $C>0$, depending only on
$\Omega$, $T$, and $M$, such that for every $\tau\ge1$,
$\omega\in\mathbb S^{n-1}$, and
$\varphi\in C_c^\infty(\mathbb R^n)$ satisfying
$\operatorname{supp}(\varphi)\cap\overline\Omega=\emptyset$,
and
$\bigl(\operatorname{supp}(\varphi)\pm T\omega\bigr)
\cap\overline\Omega=\emptyset$,
\begin{equation}\label{eq:q-after-tau}
\left|
\int_{\Omega_T}
q^{(12)}(x)\,
\varphi^3(x+t\omega)
\bigl(A_1^+(t,x)\bigr)^2
A_2^-(t,x)
\,dx\,dt
\right|
\le
C\eta^\gamma
\|\varphi\|_{H^{m+2+2N}(\mathbb R^n)}^3.
\end{equation}
\end{lemma}
\begin{proof}
    We begin by expanding the product $\bigl(v^{(1)}\bigr)^2v_0$ using $v^{(1)}=P+R$ and $v_0=Q+R_0$. Thus, we have  
    \begin{equation}\label{eq:q-expansion}
    \bigl(v^{(1)}\bigr)^2v_0 = P^2Q+2PRQ+R^2Q+P^2R_0+2PRR_0+R^2R_0. \end{equation}
    From the bounds for the amplitudes~\eqref{eq:v-w0-Linfty}, we have
   \[
\|P\|_{L^\infty(\Omega_T)}
+\|Q\|_{L^\infty(\Omega_T)}
\le
C\|\varphi\|_{H^{m+2+2N}(\mathbb R^n)},
\]
whereas \eqref{eq: r__{(2)}} and \eqref{eq: r_{(0)}}, together with $N=m+1$ give
\[
\|R\|_{L^\infty(\Omega_T)}
+\|R_0\|_{L^\infty(\Omega_T)}
\le
C\tau^{-1}
\|\varphi\|_{H^{m+2+2N}(\mathbb R^n)}.
\]
Consequently, every term on the right-hand side of
\eqref{eq:q-expansion}, except $P^2Q$, is bounded in
$L^\infty(\Omega_T)$ by
$C\tau^{-1}
\|\varphi\|_{H^{m+2+2N}(\mathbb R^n)}^3$.
We next consider the principal term. By the definitions of $P$ and $Q$,
\[
P^2Q
=
\left(
\sum_{k=0}^N\tau^{-k}m_k
\right)^2
\left(
\sum_{k=0}^N\tau^{-k}\widetilde m_k
\right),
\]
since the oscillatory factors cancel. Its leading contribution is therefore
$m_0^2\widetilde m_0$, while all the remaining terms contain at least one
negative power of $\tau$. Using the expressions for $m_0$ and
$\widetilde m_0$, we obtain
\[
m_0^2\widetilde m_0
=
\varphi^3(x+t\omega)
\bigl(A_1^+(t,x)\bigr)^2A_2^-(t,x).
\]
Hence, we have
\begin{equation}\label{eq:q-principal}
\bigl[v^{(1)}\bigr]^2v_0
=
\varphi^3(x+t\omega)
\bigl(A_1^+(t,x)\bigr)^2A_2^-(t,x)
+\mathcal R_q,
\end{equation}
where
\begin{equation}\label{eq:Rq-bound}
\|\mathcal R_q\|_{L^\infty(\Omega_T)}
\le
C\tau^{-1}
\|\varphi\|_{H^{m+2+2N}(\mathbb R^n)}^3.
\end{equation}
Combining the integral identity~\eqref{eq:q-identity} with
\eqref{eq:q-rhs-ab}, \eqref{eq:q-rhs-q2},
\eqref{eq:q-rhs-bdry}, and \eqref{eq:q-principal}, and using the
a priori bound $\|q^{(12)}\|_{L^\infty(\Omega)}\le CM$, we obtain
\begin{equation}\label{eq:q-weighted}
\begin{aligned}
\left|
\int_{\Omega_T}
q^{(12)}(x)\varphi^3(x+t\omega)
\bigl(A_1^+(t,x)\bigr)^2A_2^-(t,x)
\,dx\,dt
\right|
\le
C\left(
\frac1\tau
+\tau\zeta
+\tau^2\zeta^2
+\eta^{1/3}\tau^{2m+2}
\right)
\|\varphi\|_{H^{m+2+2N}(\mathbb R^n)}^3,
\end{aligned}
\end{equation}
where $\zeta:= \|a^{(12)}\|_{L^\infty(\Omega)} + \|b^{(12)}\|_{L^\infty(\Omega)} $. By the stability estimates obtained in the previous sections,
$\zeta\le C\eta^{\mu_b},$ for $0<\eta\leq 1$,
since $\mu_b<\mu_a$. Therefore, we have
\begin{equation}\label{eq:q-weighted-eta}
\begin{aligned}
\left|
\int_{\Omega_T}
q^{(12)}(x)\varphi^3(x+t\omega)
\bigl(A_1^+\bigr)^2A_2^-
\,dx\,dt
\right|
\le
C\left(
\tau^{-1}
+\eta^{\mu_b}\tau
+\eta^{2\mu_b}\tau^2
+\eta^{1/3}\tau^{2m+4}
\right)
\|\varphi\|_{H^{m+2+2N}(\mathbb R^n)}^3.
\end{aligned}
\end{equation}
Choose
$\tau=\eta^{-\gamma}.$
For $0<\eta<1$, this gives
$\tau^{-1}=\eta^\gamma,$
while
$\eta^{\mu_b}\tau
=
\eta^{\mu_b-\gamma}
\le \eta^\gamma,$
provided $2\gamma\le\mu_b$. Similarly,
$\eta^{2\mu_b}\tau^2
=
\eta^{2\mu_b-2\gamma}
\le \eta^\gamma,$
and
$\eta^{1/3}\tau^{2m+4}
=
\eta^{\,1/3-(2m+4)\gamma}
\le \eta^\gamma,$
since
\[(2m+5)\gamma= \frac{(2m+5)\mu_b}{2}\leq \frac{1}{3}\]
Consequently, we obtain \eqref{eq:q-after-tau}. Finally, the explicit value of $\mu_b$ gives $\dfrac{\mu_b}{2}\leq \dfrac{1}{3(2m+3)}$, so that $\gamma=\dfrac{\mu_b}{2}$, which completes the proof.
\end{proof}
We next localize the weighted estimate obtained above to derive a pointwise
bound for an attenuated half-ray transform of $q^{(12)}$. This estimate will
be the starting point for the final stability argument for the nonlinear
potential. The precise result
is stated in the following lemma.
\begin{lemma}\label{lemma:::id1-GO_3}
Let $m>n+1$ and $(a_j,b_j,q_j)\in\mathcal A(M)$ for $j=1,2$.
Assume that $T>\operatorname{diam}(\Omega)$ and $\eta<\delta^3$.
Then there exists a constant $C>0$, depending only on $\Omega$, $T$,
and $M$, such that
    \begin{equation}\label{eq::known-half-attenuated-q}
   \left| \int_0^\infty
    q^{(12)}(y-t\omega)\exp\left(-\frac12\int_0^ta_1(y-\rho\omega)\,d\rho\right)\,dt\right| \le C\eta^{\nu_1},
\end{equation} 
for every $y\in\mathbb{R}^n$ and $\omega\in\mathbb S^{n-1}$, where $ \nu_1 = \dfrac{\mu_b}{18m+26-n} $ and 
$\mu_b>0$
is the exponent determined in the stability of linear potential in the preceding section.
\end{lemma}
\begin{proof}
We begin with the weighted estimate obtained in
Lemma~\ref{lem:q-basic}. Introducing the change of variables $y=x+t\omega$ and extending $q^{(12)}$ by zero, then the above estimates \eqref{eq:q-after-tau} becomes
\begin{equation}\label{eq:q-after-change}
\left|\int_0^T\!\!\int_{\mathbb R^n}q^{(12)}(y-t\omega)
\bigl(A_1^{+}(t,y-t\omega)\bigr)^2A_2^{-}(t,y-t\omega)\,\varphi^3(y)
\,\mathrm dy\,\mathrm dt\right|\le C\eta^{\gamma}\|\varphi\|^{3}_{H^{m+2+2N}(\mathbb{R}^{n})} .
\end{equation}
Observe that $A_1^{+}A_2^{-}=\exp\Bigl(-\frac12\int_0^ta^{(12)}(y-\rho\omega)\mathrm d\rho\Bigr)$ together with
$A_1^{+}=A_1^{+}\bigl(1-A_1^{+}A_2^{-}\bigr)+\bigl(A_1^{+}\bigr)^2A_2^{-}$ and
triangle inequality, we get
\begin{align*}
&\left|\int_0^T\!\!\int_{\mathbb R^n}q^{(12)}(y-t\omega)\,
A_1^{+}(t,y-t\omega)\,\varphi^3(y)\,\mathrm dy\,\mathrm dt\right|\\
&\le\left|\int_0^T\!\!\int_{\mathbb R^n}q^{(12)}(y-t\omega)\,
A_1^{+}(t,y-t\omega)\bigl(1-\left(A_1^{+}A_2^{-}\right)(t,y-t\omega)\bigr)\varphi^3(y)\,\mathrm dy\,\mathrm dt\right|\\
&\qquad+\left|\int_0^T\!\!\int_{\mathbb R^n}q^{(12)}(y-t\omega)\,
\bigl(A_1^{+}(t,y-t\omega)\bigr)^2A_2^{-}(t,y-t\omega)\,\varphi^3(y)
\,\mathrm dy\,\mathrm dt\right| .
\end{align*}
Next, we simplify the right hand side of the above estimate. For the second
term, we use the estimate \eqref{eq:q-after-change}, and for the first
term, we first observe that
\begin{equation}\label{eq:A1A2-close-to-one}
\left|1-\left(A_1^{+}A_2^{-}\right)(t,y-t\omega)\right|
\le C\left|\int_0^ta^{(12)}(y-\rho\omega)\,\mathrm d\rho\right|
\le CT\|a^{(12)}\|_{L^\infty(\mathbb R^n)}\le C\eta^{\mu_a}\le C\eta^{\gamma},
\end{equation}
in the last step, we use $\gamma<\mu_b<\mu_a$ for $\eta\leq 1$. Now using the preceding
observation together with $\|A_1^{+}\|_{L^{\infty}(\Omega_T)}\le
e^{MT/2}$ and $\|q^{(12)}\|_{L^{\infty}(\mathbb R^n)}\le CM$, we can write
the first term as
\begin{align*}
&\left|\int_0^T\!\!\int_{\mathbb R^n}q^{(12)}(y-t\omega)\,
A_1^{+}(t,y-t\omega)\bigl(1-\left(A_1^{+}A_2^{-}\right)(t,y-t\omega)\bigr)\varphi^3(y)\,\mathrm dy\,\mathrm dt\right|\\
&\le C\eta^{\gamma}\,T\,\|q^{(12)}\|_{L^{\infty}(\mathbb R^n)}
e^{MT/2}\,\|\varphi\|^3_{L^{3}(\mathbb R^n)}
\le C\eta^{\gamma}\|\varphi\|^{3}_{H^{m+2+2N}(\mathbb{R}^{n})},
\end{align*}
where the last step follows from the Sobolev embedding
$H^{m+2+2N}(\mathbb R^n)\hookrightarrow L^3(\mathbb R^n)$. Therefore,
combining the preceding estimates
\begin{equation}\label{eq:q-A1-weighted}
\left|\int_0^T\!\!\int_{\mathbb R^n}q^{(12)}(y-t\omega)\,
\exp\Bigl(-\frac12\int_0^ta_1(y-\rho\omega)\,\mathrm d\rho\Bigr)
\varphi^3(y)\,\mathrm dy\,\mathrm dt\right|\le C\eta^{\gamma} \|\varphi\|^{3}_{H^{m+2+2N}(\mathbb{R}^{n})}.
\end{equation}
For a fixed $\omega\in\mathbb S^{n-1}$, define
\begin{equation}\label{eq:def-Psi-q}
\Psi(y,\omega):=\int_0^T q^{(12)}(y-t\omega)\exp\left(-\frac12\int_0^ta_1(y-\rho\omega)\,d\rho\right)\,dt.
\end{equation}
Then the estimate \eqref{eq:q-A1-weighted} can be written as
\begin{equation}\label{eq:Psi-q-distribution}
\left|
\int_{\mathbb R^n}
\Psi(y,\omega)\varphi^3(y)\,dy
\right|
\leq
C\eta^\gamma
\|\varphi\|_{H^{m+2+2N}(\mathbb R^n)}^3.
\end{equation}
Next, we localize the preceding estimate. Fix
$\chi\in C_c^\infty(B_1(0))$ with $0\le\chi\le1$ and
$\int_{\mathbb R^n}\chi^3\,dx=1$, and for $y_0\in\Omega_r$ and $0<h< \frac{r}{4}$ set
\begin{equation}\label{eq:def-chi-h-3}
\chi_h(y):=h^{-n/3}\chi\Bigl(\frac{y-y_0}{h}\Bigr)
\qquad\text{so that}\qquad
\supp\chi_h\subset B_h(y_0),\qquad\text{ and }\quad\int_{\mathbb R^n}\chi_h^3\,dx=1 .
\end{equation}
The support of $\chi_h$ is the same as in \eqref{Phi_delta} and $\|\chi_h\|^3_{H^{3m+4}(\mathbb R^n)}\le Ch^{-\left( 9m+12-\frac n2\right)}$ , so choosing $\varphi=\chi_{h}$, we get
\begin{equation}\label{eq:loc-split}
\Psi(y_0,\omega)=\int_{\mathbb R^n}\Psi(y,\omega)\chi_h^3(y)\,\mathrm dy
+\int_{\mathbb R^n}\bigl(\Psi(y_0,\omega)-\Psi(y,\omega)\bigr)\chi_h^3(y)\,\mathrm dy ,
\end{equation}
for $y_{0}\in \Omega_{r}$. The first integral is estimated by \eqref{eq:Psi-q-distribution}. For the second, the mean value theorem gives
\begin{equation}\label{eq:loc-second}
\bigl|\Psi(y_0,\omega)-\Psi(y,\omega)\bigr|
\le\|\nabla_y\Psi(\cdot,\omega)\|_{L^\infty(\mathbb R^n)}\,|y-y_0|\le Ch
\qquad\text{for }y\in\supp\chi_h\subset B_h(y_0),
\end{equation}
so that second integral is at most $Ch\int_{\mathbb R^n}\chi_h^3=Ch$.
Substituting \eqref{eq:Psi-q-distribution} and \eqref{eq:loc-second} into
\eqref{eq:loc-split}, we get
\begin{equation}\label{eq:loc-two-terms}
|\Psi(y_0,\omega)|\le C\bigl(h+\eta^{\gamma}h^{-\left( 9m+12-\frac n2\right)}\bigr)
\qquad \text{ for } 0<h<\tfrac r4 .
\end{equation}
The two terms are equal when $h^{9m+13-\frac n2}=\eta^{\gamma}$, that is for
$h=\eta^{\gamma/\left(9m+13-\frac n2\right)}=\eta^{\frac{\gamma}{\kappa+1}}$, and this choice satisfies
$h<\frac r4$ precisely because $\eta\le\eta_3:=(r/4)^{\frac{\kappa+1}{\gamma}}$. With it,
both terms equal $\eta^{\frac{\gamma}{\kappa+1}}$ and \eqref{eq:loc-two-terms} gives
\begin{equation}\label{eq:q-localized}
\left|\Psi(y_0,\omega)\right|\le C\eta^{\nu_1}
\qquad\text{ for }\qquad
\nu_1:=\frac{\gamma}{\kappa+1}=\frac{\mu_b}{18m+26-n}.
\end{equation}
From the preceding localization argument, for every
$y_0\in\Omega_r$ and $\omega\in\mathbb S^{n-1}$, we have the following estimate
\begin{equation}\label{eq:known-half-attenuated-q}
   \left| \int_0^\infty
    q^{(12)}(y_0-t\omega)\exp\left(-\frac12\int_0^ta_1(y_0-\rho\omega)\,d\rho\right)\,dt\right| \le C\eta^{\nu_1}
\end{equation}
By repeating the argument used in the stability analysis of the damping
coefficient $a$ and the linear potential $b$, we extend the preceding
estimate from the exterior region to the whole space. Hence, for every
$y\in\mathbb{R}^n$, we obtain the stability estimate for the attenuated half ray transform.
\begin{equation}\label{eq::;known-half-attenuated-q}
   \left| \int_0^\infty
    q^{(12)}(y-t\omega)\exp\left(-\frac12\int_0^ta_1(y-\rho\omega)\,d\rho\right)\,dt\right| \le C\eta^{\nu_1}
\end{equation} 
\end{proof}
We are now ready to complete the proof of
Proposition~\ref{prop:q-stability}. The estimate obtained in
Lemma~\ref{lemma:::id1-GO_3} provides control of the attenuated half-ray
integrals of $q^{(12)}$ from exterior starting points. We first use this
estimate to control the complete attenuated ray transform of $q^{(12)}$.
The stability of this transform, together with the a priori regularity of the
coefficients and Sobolev interpolation, then yields the desired
$L^\infty(\Omega)$-stability estimate.
 \begin{proof}[Proof of Proposition~\ref{prop:q-stability}]
With the preceding extension, we may regard $\Psi$ as defined for
every $y\in\mathbb{R}^n$ by
\begin{equation}\label{eq:Psi-global}
\Psi(y)
=
\int_0^\infty
q^{(12)}(y-t\omega)
\exp\left(
-\frac12
\int_0^t
a_1(y-\rho\omega)\,d\rho
\right)
\,dt,
\quad \text{for} \ 
y\in\mathbb{R}^n.
\end{equation}
Now to prove stability estimate for $q^{(12)}$,  we closely follow the analysis as used in \cite{bhardwaj2026reconstructionpotentialdampingcoefficients} along with previously established bound for $\Psi$.  To do so, we first define 
\begin{align}
G(t,y)
:=
q^{(12)}(y-t\omega) \ \text{  and  }\
H(t,y)
:=
\exp\left(
-\frac12\int_0^t a_1(y-\rho\omega)\,d\rho
\right),\, \mbox{for  $(t,y)\in (0,\infty)\times \mathbb{R}^n$},
\end{align}
so that $\Psi(y)= \int_0^\infty G(t,y)H(t,y)\,\mathrm dt$. We now differentiate $\Psi$ in the direction $\omega$ and obtain the following identity
\begin{align}
-\omega\cdot\nabla_y\Psi(y,\omega)
=
\int_0^\infty
\partial_tG(t,y)H(t,y)\,dt
-\frac12
\int_0^\infty
G(t,y)
\left(
a_1(y-t\omega)-a_1(y)
\right)
H(t,y)\,dt.
\label{eq:Psi-q-directional-derivative}
\end{align}
Integration by parts together with the fact that
$\partial_tH(t,y)
=
-\dfrac12
a_1(y-t\omega)H(t,y)$, we obtain
\begin{align}
\begin{split}\label{eq:Psi-q-IBP}
\int_0^\infty
\partial_tG(t,y)H(t,y)\,dt
&=
\left[
G(t,y)H(t,y)
\right]_{t=0}^{t=\infty}
-
\int_0^\infty
G(t,y)\partial_tH(t,y)\,dt\\
&=
-q^{(12)}(y)
+
\frac12
\int_0^\infty
G(t,y)a_1(y-t\omega)H(t,y)\,dt.
\end{split}
\end{align}
Above we have used the compact support of $q^{(12)}$ to obtain
$\lim_{t\to\infty}G(t,y)H(t,y)=0.$
Next, substituting \eqref{eq:Psi-q-IBP} into
\eqref{eq:Psi-q-directional-derivative}, the terms containing
$a_1(y-t\omega)$ cancel, and we obtain
\begin{align}
-\omega\cdot\nabla_y\Psi(y,\omega)
=
-q^{(12)}(y)
+
\frac12 a_1(y)\Psi(y,\omega).
\end{align}
Equivalently, we have
\begin{equation}\label{eq:q12-transport-correct}
q^{(12)}(y)
=
\omega\cdot\nabla_y\Psi(y,\omega)
+
\frac12a_1(y)\Psi(y,\omega).
\end{equation}

By the a priori regularity of the coefficients,
$\|\Psi(\cdot,\omega)\|_{W^{2,\infty}(\mathbb R^n)}
\leq C.$
Using the Landau-Kolmogorov inequality, we have
\[
\|\nabla\Psi\|_{L^\infty (\mathbb R^n)}
\leq
C
\|\Psi\|_{L^\infty (\mathbb R^n)}^{1/2}
\|\Psi\|_{W^{2,\infty}(\mathbb R^n)}^{1/2}
\]
therefore yields
\begin{equation}\label{eq:grad-Psi-correct}
\|\nabla\Psi\|_{L^\infty(\mathbb R^n)}
\leq
C
\eta^{\frac{\gamma}{2(\kappa+1)}}.
\end{equation}
Using \eqref{eq:q12-transport-correct} and
\eqref{eq:grad-Psi-correct}, we obtain
\begin{align}
\|q_1-q_2\|_{L^\infty(\Omega)}
&\leq
C
\eta^{\frac{\gamma}{2(\kappa+1)}}
+
C
\eta^{\frac{\gamma}{\kappa+1}}.
\end{align}
Since $0<\eta<1$, the first term dominates the second one.
Consequently, we have
\begin{align}\label{eq:q-final-Linfty-stability}
\|q_1-q_2\|_{L^\infty(\Omega)}
\leq
C
\eta^{\frac{\gamma}{2(\kappa+1)}}.
\end{align}
Thus, we have
\begin{equation}\label{eq:q-final-Linfty-stability-correct}
\|q_1-q_2\|_{L^\infty(\Omega)}
\leq
C\eta^{\mu_q},
\end{equation}
where
\begin{equation}\label{eq:def-gamma-q-correct}
\mu_q
:=
\frac{\gamma}
{2\left(3m+7+6N-\frac n2\right)}
=
\frac{\mu_b}
{2(18m+26-n)}
>0,
\end{equation}
the last inequality given by $\gamma=\dfrac{\mu_b}{2}$ and $N=m+1$, so that $\mu_q=\nu_1/2$. For $\eta>1$ the estimate follows from the a priori bounds and enlarging $C$. This proves the H\"older-type stability estimate for the nonlinear
 potential. 
 \end{proof}
\begin{proof}[Proof of Theorem~\ref{thm:main-stability}] 
The stability estimates established in the preceding three sections together
yield the simultaneous recovery of all three coefficients with H\"older-type
stability. Therefore, Theorem~\ref{thm:main-stability} follows, completing the
proof of the main result.
\end{proof}
\section*{Acknowledgments}

	 M.~Kumar acknowledges the support of PMRF (Prime Minister's Research Fellowship) from the government of India for his research.
	 M.~Vashisth work was supported by the ISIRD project 9--551/2023/IITRPR/10229 from IIT Ropar and  ARG-MATRICS grant ANRF/ARGM/2025/002368/MTR, from ANRF, Govt. of India.
     This work was partially supported by the FIST program of the Department of Science and Technology, Government of India, Reference No. SR/FST/MS-I/2018/22(C).\\

     \noindent\textbf{Data availability statement.} \ Data sharing is not applicable to this paper, as no datasets were generated or analyzed during the current study.\\
    
 \noindent\textbf{Conflict of interest.} \
The authors declared that they have no potential conflicts of interest with respect to the research, authorship, and/or publication of this paper.
	 \bibliography{math} 
	
	 \bibliographystyle{alpha}

\end{document}